\documentclass[a4paper,11pt]{article}

\usepackage{cmap}

\usepackage[english]{babel}

\usepackage[pdftex, colorlinks=true, bookmarksopen=true, linkcolor=red, citecolor=blue]{hyperref}

\usepackage[T1]{fontenc}
\usepackage[utf8x]{inputenc}

\usepackage{lmodern}
\usepackage{libertine}

\usepackage{bbm}
\usepackage{mathrsfs}
\usepackage{amssymb}
\usepackage{amsmath}
\usepackage{multicol}

\usepackage{amsthm}
\newtheorem*{theorem*}{Theorem}
\newtheorem{theorem}{Theorem}
\newtheorem{corollary}[theorem]{Corollary}
\newtheorem{proposition}[theorem]{Proposition}
\newtheorem{lemma}[theorem]{Lemma}

\usepackage{authblk}

\renewcommand{\epsilon}{\varepsilon}

\renewcommand{\phi}{\varphi}

\newcommand{\SetT}{\mathbb{T}}
\newcommand{\SetR}{\mathbb{R}}
\newcommand{\SetZ}{\mathbb{Z}}
\newcommand{\SetN}{\mathbb{N}}
\newcommand{\SetCl}{\mathscr{C}}

\newcommand{\SetSymmetric}{\mathscr{S}}
\newcommand{\SetPerm}{\mathfrak{S}}

\newcommand{\OpeIntd}{\,\mathrm{d}}

\DeclareMathOperator{\OpeComatrix}{Co}
\newcommand{\OpeTranspose}[1]{#1^{\mathrm{T}}}
\DeclareMathOperator{\OpeDiv}{div}
\newcommand{\OpeNablaSub}[1]{\nabla_{\!#1}}
\newcommand{\OpeKantorovich}[1]{\Psi_{\!#1}}

\newcommand{\NotZeroMeanValue}{{\diamond}}
\newcommand{\NotStarZeroMeanValue}{{*,\diamond}}
\newcommand{\NotGathered}{\widehat}
\newcommand{\NotDecomPsi}{\NotGathered{\psi}}

\newcommand{\idmatrix}{I}
\newcommand{\idmap}{\mathrm{id}}

\newcommand{\FirstOperator}{\mathcal{F}}
\newcommand{\SecondOperator}{\mathcal{G}}
\newcommand{\InverseOperator}{\mathcal{S}}

\title{ From Knothe's Rearrangement \\ to Brenier's Optimal Transport Map }

\author{Nicolas \textsc{Bonnotte}}
\affil{Universit\'e Paris-Sud, Orsay}
\affil{Scuola Normale Superiore, Pisa}

\date{ \today \\ \small (revised version) }

\hypersetup{pdfauthor={ Nicolas Bonnotte },pdftitle={ From Knothe's Rearrangement to Brenier's Optimal Transport Map }}

\begin{document}

%
%

\maketitle

%
%

\begin{abstract}
   The Brenier optimal map and the Knothe--Rosenblatt rearrangement are two instances of a transport map, that is to say a map sending one measure onto another.  The main interest of the former is that it solves the Monge--Kantorovich optimal transport problem, while the latter is very easy to compute, being given by an explicit formula.
         
   A few years ago, Carlier, Galichon, and Santambrogio showed that the Knothe rearrangement could be seen as the limit of the Brenier map when the quadratic cost degenerates. 
   In this paper, we prove that on the torus (to avoid boundary issues), when all the data are smooth, the evolution is also smooth, and is entirely determined by a \textsc{pde} for the Kantorovich potential (which determines the map), with a subtle initial condition. The proof requires the use of the Nash--Moser inverse function theorem. 

   This result generalizes the \textsc{ode} discovered by Carlier, Galichon, and Santambrogio when one measure is uniform and the other is discrete, and could pave to way to new numerical methods for optimal transportation.

\bigskip


\medskip

  \noindent \textbf{\textsc{ams} classification.}\quad \textsc{35j96, 47j07, 49k21}.
\end{abstract}

%
%

\vfill

\thispagestyle{empty}

\footnotesize 

\noindent \rule{7em}{0.4pt}
\setlength{\parindent}{1em}

\smallskip

Nicolas \textsc{Bonnotte}

\smallskip

Département de mathématiques

Université Paris-Sud

91405 Orsay \textsc{cedex}

\textsc{France}

\smallskip

\url{http://www.normalesup.org/~bonnotte} 

\url{nicolas.bonnotte@normalesup.org}

\normalsize

\vspace*{-3.5em}

\thispagestyle{empty}

\newpage

%
%

\section{Introduction}

Although optimal transport theory has far-reaching applications, in fields as diverse as continuum mechanics, statistics or image processing, its underlying problem is quite simple: how to send one probability measure onto another, while minimizing some cost of transportation? 
Let us denote by $\mu$ and $\nu$ those two measures, defined respectively on $X$ and $Y$. They could for instance represent the respective distributions of some goods  being produced, and the needs for them, and the problem would then be to determine how to organize the supply so that the total cost of transportation is as small as possible. 

What we are looking for is a map $T : X \rightarrow Y$ telling us where to send what is in $x$; but $T$ will be suitable only if, for any measurable set $A \subset Y$, 
the goods sent by $T$ in $A$ match the needs of the same region,
that is to say if $\mu(T^{-1}(A)) = \nu(A)$. If this condition is satisfied, $\nu$ is said to be the push-forward of $\mu$ by $T$, and we write $\nu = T \# \mu$. Let us denote by $c(x,y)$ the cost for going from $x$ to $y$, then the total cost of transportation we want to minimize is
\begin{equation}\label{EqMonge}
\int_X c(x,T(x)) \OpeIntd \mu(x).
\end{equation}
Notice however that an optimal map may well not exist, and worse, there might even be no map transporting $\mu$ onto $\nu$ at all, e.g. if $\mu$ is discrete and $\nu$ is uniform.

The problem of finding a map $T$ minimizing \eqref{EqMonge}, and such that $\nu = T \# \mu$, was first studied by Monge~\cite{MONGE01} in the 18th century. 
In the 1940s, Kantorovich \cite{KANTOROVICH01} introduced 
the following relaxation of Monge's problem: instead of sending all that is in $x$ to a unique destination $y = T(x)$, he allowed some splitting. Any strategy for sending $\mu$ onto $\nu$ can then be represented by a measure $\gamma$ on $X \times Y$, such that $\gamma(A \times B)$ gives the share of the goods to be moved from $A$ to $B$. A plan $\gamma$ is suitable if it matches the production and the needs, i.e. if
\[\gamma(A \times Y) = \mu(A), \qquad \gamma(X \times B) = \nu(B).\]
This simply means that $\mu$ and $\nu$ must be the marginals of $\gamma$. Let us denote by $\Gamma(\mu,\nu)$ the set of all such suitable plans.  The total cost of transportation with the plan $\gamma $ is
\begin{equation}\label{EqMongeKantorovich}
\int_{X \times Y} c(x,y) \mathrm{d}\gamma(x,y).
 \end{equation}
The Monge--Kantorovich problem consists in finding $\gamma \in \Gamma(\mu,\nu)$ minimizing \eqref{EqMongeKantorovich}. It is indeed a relaxation of the Monge problem, since if $T : X \rightarrow Y$ sends $\mu$ onto $\nu$, then the push-forward $\gamma := (\idmap,T) \# \mu $ of $\mu$ by $x \mapsto (x,T(x))$ is in $\Gamma(\mu,\nu)$, and the costs \eqref{EqMonge} and \eqref{EqMongeKantorovich} are equal.

At the end of the 1980s, Brenier~\cite{MR923203, MR1100809} discovered the optimal transport map for the Monge problem to exist as the gradient of a convex function and to be unique, at least when $X = Y = \SetR^N$, for the cost $c(x,y) = \frac{1}{2}|x-y|^2$, if $\mu$ is absolutely continuous and if $\mu$ and $\nu$ have finite second order moments. His result was then extended to measures defined on the torus $\SetT^N$ by Cordero-Erausquin~\cite{MR1711060}, or more generally on a Riemannian manifold by McCann \cite{MR1844080}. While on $\SetR^N$ the optimal map is $T(x) = \nabla \phi(x)$ with $\phi$ convex, on the torus $\SetT^N$ the optimal map can be written as $T(x) = x - \nabla \psi(x)$ with $\psi : \SetT^N \rightarrow \SetR^N$ such that $\phi : x \mapsto \frac{1}{2}x^2 - \psi(x)$ defines a convex function on $\SetR^N$. More generally, on a Riemannian manifold $T(x) = \exp_x(-\nabla \psi(x))$ for some map $\psi$, called the Kantorovich potential because it is linked to a dual formulation for the relaxed problem.

\medskip

Being able to compute the optimal map $T$, or the underlying potential $\psi$, is obviously of huge interest. When the measures are discrete, if there is a solution to the Monge problem, it can be obtained for instance with the auction algorithm. In the continuous case, the solution is also easy to compute in dimension $1$, for if $\mu$ and $\nu$ are absolutely continuous, and if $F$, $G$ stand for their respective cumulative distributions, i.e. $F(x) := \mu((-\infty,x])$ and $G(y) := \nu((-\infty,y])$,
then the optimal transport map is $T = G^{-1} \circ F$. 

Unfortunately when the dimension is $N > 1$, there is no such easy formula, and it is much more complicated to compute Brenier's map---although not impossible. Among the most notable methods, we could cite the one due to Benamou and Brenier~\cite{MR1738163}, relying on a dynamic formulation of the Monge--Kantorovich problem, in which one tries to minimize the average kinetic energy of the particles during their transportation. On the other hand, Angenent, Haker, and Tannenbaum \cite{MR2001465} proposed a steepest descent method, starting from a transport map (for instance, the Knothe--Rosenblatt rearrangement) and letting it evolve so as to reduce the associated transport cost.
A couple of years later, Leoper and Rapetti~\cite{MR2121899} used the characterization of the optimal transport map through the existence of a convex potential, to compute Brenier's map starting from any potential and, with a Newton algorithm, altering it so as to finally get the optimal potential. 

Our hope here is that the results presented in this paper might lead to yet another approach for computing Brenier's map. Our starting point is a direct connexion, proved by Carlier, Galichon, and Santambrogio~\cite{MR2607321} a few years ago but hinted beforehand by Brenier, between the optimal transport map and the Knothe--Rosenblatt rearrangement.  This leads us to believe it might be possible to compute Brenier's map starting from the rearrangement (as in the paper by Angenent, Haker, and Tannenbaum \cite{MR2001465}), and then proceeding with a continuation method (as in the work by Loeper and Rapetti~\cite{MR2121899}).

\medskip

This so-called Knothe--Rosenblatt rearrangement, which is also built so as to send one measure onto another, was first introduced by Rosenblatt~\cite{MR0049525} and Knothe~\cite{MR0083759}\footnote{Interestingly, Knothe used this rearrangement to prove the isoperimetric inequality, for which it is well suited\dots but in fact not as much as Brenier's map, which Figalli, Maggi, and Pratelli \cite{MR2672283} used more recently to prove sharp isoperimetric inequalities.}.  It can be defined for absolutely continuous probability measures on $\SetR^2$ or on $\SetT^2 = \SetR^2/\SetZ^2$  (in higher dimension, the construction is analogous) as follows: To begin with, let us denote by $f$, $g$ the densities of $\mu$, $\nu$. Then, take the first marginals, which we denote by $\mu^1$and $\nu^1$; their respective densities are
\[ f^1(x_1) = \int f(x_1,x_2) \OpeIntd x_2 \qquad \text{and} \qquad g^1(y_1) = \int g(y_1,y_2)\OpeIntd y_2.\]
Define $R^1$ as the optimal transport map between $\mu^1$ and $\nu^1$. Next, consider the disintegration of $\mu$ and $\nu$ with respect to $\mu^1$ and $\nu^1$, that is to say the two family of probabilities measures $\{ \mu^2_{x_1}\}$ and $\{\nu^2_{y_1}\}$ such that the densities of $\mu^2_{x_1}$ and $\nu^2_{y_1}$ are
\[ f^2_{x_1} (x_2) = \frac{f(x_1,x_2)}{f^1(x_1)} \qquad \text{and} \qquad g^2_{y_1}(y_2) = \frac{g(y_1,y_2)}{g^1(y_1)}.\]
For any $x_1$ let $R^2(x_1,\cdot)$ be the optimal transport map between $\mu^2_{x_1}$ and $\nu^2_{R^1(x_1)}$. Then the rearrangement is $R(x_1,x_2) := (R^1(x_1),R^2(x_1,x_2))$. It is not difficult to check that it sends $\mu$ onto $\nu$.

What Carlier, Galichon and Santambrogio proved is that, if in Monge's problem the cost is, for instance, replaced with
\[ c_t(x,y) = \frac{1}{2} \sum_{k=1}^d t^{k-1} |x_k-y_k|^2,\]
then, when the two measures are absolutely continuous, as $t$ goes to $0$, the corresponding optimal transport maps $T_t$ converge in $L^2$ to the rearrangement $R$. When the initial measure $\mu$ is uniform and the final measure $\nu$ is discrete, $\nu = \sum a_i \delta_{y_i}$, they could also establish an
\textsc{ode} governing the evolution of the Kantorovich potential $\psi_t$, at least when the first coordinates of the $y_i$ are distinct.

Thus, the following questions arise: in the continuous case, is it also possible to find a differential equation satisfied by $\psi_t$? and if the answer is yes, is there uniqueness, that is to say, given the proper initial condition for $t=0$, is $\psi_t$ the only solution to this equation? As we are going to see, the answer to both question is positive, at least, to discard boundary issues, on the torus. More precisely, we have the following:

\begin{theorem*}
Let $A_t$ be the $(1,t,\ldots,t^{N-1})$ diagonal matrix, and $\mu, \nu$ two probability measures on $\SetT^N = \SetR^N/\SetZ^N$ with smooth, strictly positive densities $f,g$. The optimal transport map for the cost
\[ c_t(x,y) = \frac{1}{2} \sum_{k=1}^N t^{k-1} d(x_k,y_k)^2,\]
$d$ standing for the usual distance on $\SetT^1$, is then $T_t(x) = x-A_t^{-1}\nabla \psi_t(x)$, where the Kantorovich potential $\psi_t$ is chosen so that $\int \psi_t = 0$. The map $t \mapsto \psi_t$ is smooth from $(0,+\infty)$ to $\SetCl^\infty(\SetT^N)$, 
with $A_t - D^2 \psi_t > 0$ at all times, and satifies
\begin{equation}\label{EqIntro}
 \OpeDiv\left( f \left[ \idmatrix - A_t^{-1}D^2 \psi_t \right]^{-1} \left( A_t^{-1}\nabla \dot{\psi}_t - A_t^{-1} \dot{A}_t A^{-1}_t \nabla \psi_t \right)\right) = 0.
 \end{equation}
Moreover, $\psi_t$ is the unique solution of \eqref{EqIntro} such that, if we write for $t \neq 0$,
\begin{equation}\label{EqDecomp}
\psi_t(x_1,\ldots,x_N) = \psi_t^1(x_1) +  t\psi_t^2(x_1,x_2) + \ldots + t^{N-1} \psi^N_t(x_1,\ldots,x_N),
\end{equation}
with 
\[ \forall x_1, \ldots, x_{k-1}, \qquad \int \psi^k_t(x_1,\ldots,x_{k-1},z) \OpeIntd z = 0,\]
then
$t  \mapsto (\psi^1_t,\ldots,\psi^N_t)$ 
is $\SetCl^2$ on $[0,\infty)$, and the Knothe--Rosenblatt rearrangement $R=(R^1,\ldots,R^N)$ is given by
\[ R^k(x_1,\ldots,x_k) = x_k - \partial_k \psi^k_0(x_1,\ldots,x_k).\]
\end{theorem*}

The first point is obtained by noticing that,  at least when $t$ stays away from $0$, $\psi_t$ is the unique solution to a Monge--Ampère equation $\FirstOperator(A_t,\psi_t) = 0$, where
\[ \FirstOperator(A,u)(x) := f(x) - g\left(x - A^{-1} \nabla u(x) \right) \det\left(\idmatrix - A^{-1} D^2 u(x) \right),\]
is defined on a proper subset of $\SetSymmetric^{++}_N \times \SetCl^2(\SetT^N)$,
and then proving that we can apply the implicit function theorem. As it is well-known, the invertibility of the differential $D_u \FirstOperator$ in $(A_t,\psi_t)$ is equivalent to the existence and uniqueness of the solution to a strictly elliptic equation, so the argument is rather straightforward.

For small times, because of the degeneracy of $A_t$, we need the decomposition~\eqref{EqDecomp}, which leads us to introduce another operator, namely:
\[\SecondOperator(t,u^1,u^2,\ldots,u^N):=\FirstOperator\left(A_t, \sum t^{k-1} u^k \right),\] 
defined on a good subset of $[0,+\infty) \times \SetCl^2(\SetT^1) \times \SetCl^2(\SetT^2) \times \cdots \times \SetCl^2(\SetT^N)$, 
in such a way that $(\psi^1_t,\ldots,\psi^N_t)$ is the only $(u^1,\ldots,u^N)$ such that $\SecondOperator(t,u^1,\ldots,u^N)=0$. Unfortunately, a loss of regularity for the solutions $(v^1,\ldots,v^N)$ of the equation
\[D_u \SecondOperator(t,\psi^1,\ldots,\psi^N)(v^1,\ldots,v^N) = q\]
prevents us from applying the implicit function theorem once more. We circumvent this difficulty by using the smoothness of the Kantorovich potential $\psi_t$, which allows us to  define $\mathcal{G}$ on a subset of $[0,+\infty) \times \SetCl^\infty(\SetT^1) \times \cdots \times \SetCl^\infty(\SetT^N)$, so that to have an infinite source of regularity, and then use the Nash--Moser version of the implicit function theorem. 

\medskip

We do not know if there is an equivalent result on $\SetR^N$. To be able to construct the Knothe rearrangement, compactness is required, but in $\SetR^N$ this comes with a boundary. The problem is that the rearrangement is more easily contented with sets whose shapes are somewhat compatible with the axes, e.g. the square, but known regularity results for Brenier's map fail to apply in that kind of setting.

%
%

\paragraph*{Acknowledgements}
This work is part of a \textsc{p}h\textsc{d} thesis supervised by Luigi Ambrosio (\textsc{sns}, Pisa) and Filippo Santambrogio (Univ. Paris--Sud), whom the author would like to thank warmly for their advice and strong support.
Financial support is provided in part by a “Vinci” grant from the Franco--Italian University.
Much of this paper is also the result of an extended stay in Pisa in Fall 2011, which was made possible thanks to the \textsc{ens}--\textsc{sns} exchange program.

%
%

\section{General quadratic costs on the torus}

Given two probability measures $\mu$, $\nu$ on the torus $\SetT^N = \SetR^N/\SetZ^N$, we want to study the evolution with $t$ of the optimal transport map for the cost
\begin{equation}\label{EqCostDefinition}
 c_t(x,y) = \frac{1}{2} \sum_{k=1}^N \left( \prod_{i < k} \lambda_i(t) \right) d(x_k,y_k)^2,
 \end{equation}
where $d : \SetT^1 \times \SetT^1 \rightarrow [0,+\infty)$ is the usual distance on $\SetT^1$, the $\lambda_i : \SetR \rightarrow [0,+\infty)$ are smooth and such that $\lambda_i(t)=0$ if and only if $t=0$.  For $t>0$, this is a kind of quadratic cost on the torus.
Notice that, more generally, we can define a cost given any positive-definite symmetric matrix $A \in \SetSymmetric_N^{++}$ as follow: first, consider $\tilde{c} : \SetR^N \times \SetR^N \rightarrow [0,+\infty)$ defined by
\[ \tilde{c}(x,y)^2 := \inf_{k \in \SetZ^N}  \frac{1}{2} A(x-y-k)^2, \] 
where $Az^2$ is a convenient shorthand for $\langle Az, z \rangle$, and then take the induced map $c : \SetT^N \times \SetT^N \rightarrow [0,+\infty)$. This is equivalent to changing the usual metric on $\SetT^N$ with the one induced by $A$ in the canonical set of coordinates, and then taking half the resulting squared distance as the cost.

An interesting property of such a cost $c$ is that in this case the so-called $c$-transform of a function $u : \SetT^N \rightarrow \SetR$ is strongly connected to the Legendre transform (for the scalar product induced by $A$) of $x \mapsto \frac{1}{2}Ax^2 - u(x)$, defined on $\SetR^N$ (we then see $u$ as a periodic function on $ \SetR^N$).  Let us recall that the $c$-transform of $u$ is the map $u^c : \SetT^N \rightarrow \SetR$ defined by
\[ u^c(y) = \inf_{x \in \SetT^N} \left\{ c(x,y) - u(x) \right\}.\]
This is interesting, because McCann~\cite{MR1844080} showed that, under suitable assumptions, the optimal transport map $T$ can be written as $T(x) = \exp_x(-\nabla \psi(x))$, for some function $\psi$ such that $\psi^{cc} = \psi$. A map $u$ such that $u^{cc} = u$ is called $c$-concave.

\begin{lemma}\label{LemCConcave}
A function $u : \SetT^N \rightarrow \SetR$ is $c$-concave if and only if
\[ v: \left\{ 	\begin{array}{ccl}
			\SetR^N & \rightarrow & \SetR \\
			x & \mapsto & \frac{1}{2}Ax^2 - u(x)
		\end{array}	\right.\]
is convex and lower semi-continuous. If $u$ is $\SetCl^2$ and such that $A - D^2u > 0$, then $x \mapsto x - A^{-1}\nabla u(x)$ induces a diffeomorphism $\SetT^N \rightarrow \SetT^N$.
\end{lemma}

\begin{proof}
If $u$ is $c$-concave, then $v$ is convex and lower semi-continuous, for it can be written as a Legendre transform:
\begin{align*}
 v(x) 		& = \frac{1}{2}Ax^2 - u^{cc}(x) \\
 		& = \frac{1}{2}Ax^2 - \inf_{y \in \SetT^N} \left\{ c(x,y) - u^c(y) \right\} \\
		& = \sup_{y \in \SetR^N} \left\{ \frac{1}{2}Ax^2 - \tilde{c}(x,y) + u^{c}(y)\right\} \\
		& = \sup_{y \in \SetR^N} \sup_{k \in \SetZ^2} \left\{ \frac{1}{2}Ax^2 - \frac{1}{2}A(x-y-k)^2 + u^{c}(y)\right\} \\
		& = \sup_{y \in \SetR^N} \left\{ \langle A x, y \rangle - \left[ \frac{1}{2} A y^2 - u^{c}(y)\right] \right\}.
\end{align*}
Conversely, if $v$ is convex and lower semi-continuous, then it is equal to its double $A$-Legendre transform:
\[ v(x) = \sup_{y \in \SetR^N} \left\{ \langle Ax, y \rangle - \sup_{z \in \SetR^N} \left[ \langle Az, y \rangle - v(z) \right] \right\}.\]
Therefore,
\begin{align*}
 u(x) & = \frac{1}{2} Ax^2 -  \sup_{y \in \SetR^N} \left\{ \langle Ax, y \rangle - \sup_{z \in \SetR^2} \left[ \langle Az, y \rangle - v(z) \right] \right\} \\
 		& = \inf_{y \in \SetR^N} \left\{ \frac{1}{2} A(x-y)^2 -  \frac{1}{2}Ay^2 + \sup_{z \in \SetR^2} \left[ \langle Az, y \rangle - v(z) \right] \right\} \\
		& =  \inf_{y \in \SetR^N} \left\{ \frac{1}{2} A(x-y)^2 - \inf_{z \in \SetR^2} \left[ \frac{1}{2}A(y-z)^2 - u(z) \right] \right\},
\end{align*}
that is to say $u(x)=u^{cc}(x)$.

If $u$ is $\SetCl^2$ and such that $A - D^2 u > 0$, then by compactness $A - D^2 u \geq \epsilon \idmatrix$ for some $\epsilon > 0$. Thus, $v$ being convex with a super-linear growth,  $\nabla v : \SetR^N \rightarrow \SetR^N$ is a diffeomorphism, and so is the map $T:x \mapsto x - A^{-1}\nabla u(x)$. Notice that, if $k \in \SetZ^N$, then $T(x+k) = T(x) + k$, therefore $T$ induces a diffeomorphism $\SetT^N \rightarrow \SetT^N$.
\end{proof}

In the next proposition, we start from the existence and uniqueness of the Kantorovich potential for such a generalized cost (this comes from McCann~\cite{MR1844080}), and then apply the results of Caffarelli \cite{MR1124980} to get its smoothness, in the exact same way as Cordero-Erausquin~\cite{MR1711060} did.  More general results regarding the regularity of the potential, and thus, of the optimal transport map, on arbirary products of spheres have been recently obtained by Figalli, Kim, and McCann~\cite{FKMC2011}.

\begin{proposition}\label{ProKantorovich}
  Let $\mu$ and $\nu$ be two probability measures on $\SetT^N$ with smooth, strictly positive densities, and let $c$ be the quadratic cost on $\SetT^N \times \SetT^N$ induced by a definite-positive symmetric matrix $A$.

  Then there is a unique $c$-concave function $\psi : \SetT^N \rightarrow \SetR$ with $\int \psi = 0$ such that $T : \SetT^N \rightarrow \SetT^N$ defined by $T(x) := x - A^{-1}\nabla \psi(x)$ sends $\mu$ onto $\nu$.
  
  The function $\psi$ is a Kantorovich potential, it is smooth, and the application $\phi : x \mapsto \frac{1}{2}A x^2 - \psi(x)$ is a smooth, strictly convex function on $\SetR^N$. 
  
  The transport map $T$ is optimal for the cost $c$.  There is no other optimal transport plan but the one it induces.
\end{proposition}

Of course in this proposition, instead of $T(x) := x - A^{-1}\nabla \psi(x)$ we should have written $T(x) = x - \pi(A^{-1}\nabla \psi(x))$, where $\pi : \SetR^N \rightarrow \SetT^N$ is the usual projection. 

\begin{proof}
Let us denote by $\OpeNablaSub{A}$ the gradient for the metric induced by $A$.
Then according to McCann \cite{MR1844080}, there is a Lipschitz function $\psi : \SetT^N \rightarrow \SetR$ that is $c$-concave
and such that $T:x \mapsto \exp_x[-\OpeNablaSub{A} \psi(x)]$ pushes $\mu$ forward to $\nu$. It is uniquely defined if the condition $\int \psi(x) \OpeIntd x = 0$ is added, and moreover it is optimal for the Monge--Kantorovich problem. Notice that  here on the torus, $\exp_x [-\OpeNablaSub{A} \psi(x)] = x - A^{-1} \nabla \psi(x)$. 

For any $x \in \SetR^N$, let $\phi(x) := \frac{1}{2}A x^2 - \psi(x)$. Then $T(x)=A^{-1} \nabla \phi(x)$ sends $\mu$ onto $\nu$, seen as periodic measures on $\SetR^N$. 
Moreover, according to Lemma~\ref{LemCConcave}, $\phi$ is a convex function. Now, let $V$ be a open, convex subset of $\SetR^N$, and define $U = (\nabla \phi)^{-1}(V)$; then $\nabla \phi$ sends $\mu{|_U}$ onto $A \# \nu{|_V}$, and both measures are still absolutely continuous with smooth, bounded, strictly positive densities. Therefore we are entitled to apply the results of Caffarelli \cite{MR1124980}, and thus we get that $\phi$ is strictly convex and smooth on $U$. As $U$ is arbitrary, $\phi$ is strictly convex and smooth on $\SetR^N$. Thus, $\psi$ is also smooth, and $T$ is a diffeomorphism.
 \end{proof}


\section{PDE satisfied for positive times}\label{SecNonDegenerate}

Let $\mu$ and $\nu$ be two probability measures on $\SetT^N$ with smooth, strictly positive densities $f$ and $g$. According to Proposition \ref{ProKantorovich}, for any $A \in \SetSymmetric_N^{++}$, we have a smooth Kantorovich potential $\OpeKantorovich{A} : \SetT^N \rightarrow \SetR$. What can we say of the regularity of $\Psi : A \mapsto \OpeKantorovich{A}$?

As $x \mapsto x - A^{-1}\nabla \OpeKantorovich{A}(x)$ sends $\mu$ onto $\nu$, the following Monge--Ampère equation is satisfied:
\[ f(x) = g\left(x - A^{-1}\nabla \OpeKantorovich{A}(x) \right) \det\left(\idmatrix - A^{-1}D^2 \OpeKantorovich{A}\right).\]
For $u \in \SetCl^2(\SetT^N)$ such that $A - D^2 u > 0$ and $\int u = 0$, we set
\[ \FirstOperator(A,u) = f - g\left(\idmap - A^{-1} \nabla u\right) \det\left(\idmatrix - A^{-1} D^2 u\right). \]
Thanks to the characterization of $c$-concave functions from Lemma \ref{LemCConcave}, and to Proposition \ref{ProKantorovich}, we have

\begin{lemma}\label{LemCharacterizationAPriori1}
For any $u \in \SetCl^2(\SetT^N)$ such that $A - D^2 u > 0$ and $\int u = 0$, we have $\FirstOperator(A,u) = 0$ if and only if $u = \Psi_A$.
\end{lemma}

We are now going to prove that we can apply the implicit function theorem.

\medskip

In the following, for any function space $X$ we denote with a $\NotZeroMeanValue$ subscript the space formed by the elements of $X$ having a zero mean value, e.g. $\SetCl^2_\NotZeroMeanValue$ is the space of all $u \in \SetCl^2$ such that $\int u = 0$. 

\begin{lemma}\label{LemDerivativeFirstOperator}
The operator $\mathcal{F}$ is smooth. For any $A \in \SetSymmetric_N^{++}$, if $u \in \SetCl^2_\NotZeroMeanValue(\SetT^N)$ is such that $A - D^2 u > 0$, if $v \in \SetCl^2_\NotZeroMeanValue(\SetT^2)$, then
\begin{align*}
 D_u \FirstOperator(A,u)v 		& 	=  	 \OpeDiv \left(\left(f-\FirstOperator(A,u)\right) \left[A - D^2  u\right]^{-1} \nabla v\right) \\ 
 						& 	= 	\frac{1}{\det A}\OpeDiv \left(g\left(\idmap - A^{-1} \nabla u\right) \OpeTranspose{\left[\OpeComatrix{\left(A - D^2 u\right)}\right]} \nabla v\right).
 \end{align*}
\end{lemma}

We denote by $\OpeTranspose{M}$ the transposed matrix of $M$, and by $\OpeComatrix(M)$ its comatrix, that is to say the matrix formed by the cofactors.

\begin{proof}
 The smoothness of $\mathcal{F}$ is obvious. By substitution, for any $\xi \in \SetCl^\infty$,
\[ \int \xi\left(x-A^{-1}\nabla u(x)\right)\left[f(x) - \FirstOperator(A,u)(x)\right] \OpeIntd x = \int \xi(y) g(y) \OpeIntd y.\]
Therefore, if we conveniently set $T_Au(x) := x - A^{-1} \nabla u(x)$ and differentiate the previous equation with respect to $u$ along the direction $v$, we get
\[- \int \left\langle \nabla \xi(T_Au), A^{-1}\nabla  v \right\rangle \left(f-\FirstOperator(A,u)\right) - \int \xi(T_Au) D_u \FirstOperator(A,u) v = 0.\]
Since $\nabla [\xi \circ T_Au] = \OpeTranspose{[D T_Au]} \nabla \xi(T_Au)$, we have
\begin{align*}  
	\left\langle \nabla \xi(T_Au), A^{-1}\nabla  v \right\rangle 	& 	= \left\langle \nabla [\xi \circ T_Au], [D T_Au]^{-1} A^{-1} \nabla v \right\rangle \\
												&	= \left\langle \nabla [\xi \circ T_Au], [\idmatrix - A^{-1} D^2 u]^{-1} A^{-1} \nabla v \right\rangle,
\end{align*} 
and this yields
\begin{multline*}
 \int \xi(T_Au) D_u \FirstOperator(A,\psi_A) v \\
 = \int \xi(T_Au) \OpeDiv\left( \left(f-\FirstOperator(A,u)\right)[\idmatrix - A^{-1} D^2 u]^{-1} A^{-1} \nabla v \right),
 \end{multline*}
and thus, since $\xi \circ T_Au$ is arbitrary, we get the first equality. Then, we can easily obtain the second expression using the formula $M^{-1} = \OpeTranspose{[\OpeComatrix{M}]}/\det(M)$.
\end{proof}


\begin{lemma}\label{LemEUFirstOrder}
Let $\epsilon > 0$  and $A \in \SetSymmetric_N^{++}$. If $u \in \SetCl^{2}_\NotZeroMeanValue(\SetT^N)$ is such that 
\[A - D^2 u > \epsilon(\det A)^{\frac{1}{N-1}} \idmatrix,\] 
then for any $q \in [H_\NotZeroMeanValue^1(\SetT^N)]^*$, there is a unique $v \in H^{1}_\NotZeroMeanValue(\SetT^N)$ such that  
\begin{equation}\label{EqDiffSM}
	D_u \FirstOperator(A,u)v = q.
\end{equation}
Moreover, $\| v \|_{H^{1}} \leq C_{\epsilon} \|q\|_{(H^1_\NotZeroMeanValue)^*}$. 
\end{lemma}

\begin{proof}
As $A-D^2 u > \epsilon (\det A)^{1/(N-1)} \idmatrix$, the lowest eigenvalue of $\OpeComatrix{(A-D^2u)}$ is bounded by $\epsilon^{N-1} \det A$. Since $g > \delta$ for some $\delta > 0$,  for any $\xi \in \SetCl^\infty(\SetT^N)$,
\begin{align*}
\epsilon^{N-1} \det A \int |\nabla \xi|^2 
	&\leq \int \langle \OpeTranspose{[\OpeComatrix{(A-D^2u)}]} \nabla \xi, \nabla \xi \rangle \\
	& \leq \frac{1}{\delta} \int  g\left(\idmap - A^{-1}\nabla u\right) \langle  \OpeTranspose{[\OpeComatrix{(A-D^2u)}]} \nabla \xi, \nabla \xi \rangle,
\end{align*}
and thus
\begin{equation}\label{EqCoercitivity}
  \int |\nabla \xi|^2  \leq -\frac{1}{\delta\epsilon^{N-1}} \int \xi D_u \FirstOperator(A,u)\xi.
\end{equation}
Therefore, thanks to the existence of a Poincaré inequality on $H^1_\NotZeroMeanValue(\SetT^N)$, the map $(\xi,\eta) \mapsto \int \eta D_u \FirstOperator(A,u)\xi$ induces a coercive, continuous bilinear form on $H^1_\NotZeroMeanValue$. 
We are thus entitled to apply the Lax--Milgram theorem, which yields the existence and the uniqueness, for every $q \in (H^1_\NotZeroMeanValue)^*$,  of a $v \in H^1_\NotZeroMeanValue$ satisfying \eqref{EqDiffSM}. 
Moreover, \eqref{EqCoercitivity} immediately gives us $\| v\|_{H^{1}} \leq \frac{1}{\delta\epsilon^{N-1}}\|q\|_{(H^1_\NotZeroMeanValue)^*}$.
\end{proof}

The regularity of the solutions to an elliptic equation is well known. However, as in the following we will need some very precise estimates to apply the Nash--Moser theorem, let us give a proof of the following result:

\begin{lemma}\label{LemEUNextOrders}
Under the same assumptions, and with the same notations, for any $n \geq 1$, if $u \in \SetCl^{n+2}_\NotZeroMeanValue$ and $q \in H_\NotZeroMeanValue^{n-1}$ satisfy $\|u\|_{\SetCl^{3}} + \|q\|_{(H^1_\NotZeroMeanValue)^*} \leq M$, then $v \in H^{n+1}_\NotZeroMeanValue$, and
\begin{equation} \label{EqBoundednessInduction}
 \| v\|_{H^{n+1}} \leq C_{\epsilon, M, n}\left\{ \|q\|_{H^{n-1}} + \| u \|_{\SetCl^{n+2}} \right\}.
 \end{equation}
\end{lemma}

\begin{proof}
We proceed by induction. Let $n \geq 1$, $u \in \SetCl^{n+2}_\NotZeroMeanValue$ and $q \in H^{n}_\NotZeroMeanValue$ such that $A - D^2 u > \epsilon (\det A)^{1/(N-1)} \idmatrix$ and $\|u\|_{\SetCl^{3}} + \|q\|_{(H^1_\NotZeroMeanValue)^*} \leq M$. We assume that we already know that the corresponding solution $v$ is in  $H^{n}_\NotZeroMeanValue$, and that
\begin{equation} \label{EqBoundednessInductionPreviousOrder}
 \| v\|_{H^{n}} \leq C_{\epsilon, M, n-1}\left\{ \|q\|_{H^{n-2}} + \| u \|_{\SetCl^{n+1}} \right\}.
 \end{equation}
Notice that we do have such an inequality for $n=1$, according to the previous lemma, but with $\|q\|_{(H^1_\NotZeroMeanValue)^*}$ instead of $\|q\|_{H^{-1}}$. Let us now show that it implies $v \in H^{n+1}_\NotZeroMeanValue$ and
\[  \| v\|_{H^{n+1}} \leq C_{\epsilon, M, n}\left\{ \|q\|_{H^{n-1}} + \| u \|_{\SetCl^{n+2}} \right\}.\]

First, we set $B_A u := (f-\FirstOperator(A,u))[A-D^2u]^{-1}$, so that Equation \eqref{EqDiffSM} becomes
\begin{equation}\label{EqDiffSM2}
 D_u \FirstOperator(A,u)v = \OpeDiv(B_A u \nabla v).
 \end{equation}
Then, for $h \in \SetR^2$ and $\xi \in H^1$, we also define
\[\tau_h \xi(x) := \xi(x+h) \qquad \text{and} \qquad \delta_h \xi(x) := \frac{\xi(x+h)-\xi(x)}{h}.\]
Notice then that $\delta_h (\eta\xi) = \eta \delta_h \xi + (\delta_h \eta) \tau_h \xi$, and $\| \delta_h \xi\|_{L^2} \leq \|\xi\|_{H^1}$.

Let $\nu \in \SetN^2$ be a $2$-index, with $|\nu|:=\nu_1+\nu_2=n-1$, and let $h \in \SetR^2$ be small enough. We can apply the operator $\delta_h$ to Equation \eqref{EqDiffSM2}, and we then obtain
\[ \OpeDiv(B_Au \nabla \delta_h v) = \delta_h q - \OpeDiv\left[(\delta_hB_Au)\nabla \tau_h v\right]\]
Then, by applying $\partial_\nu$, we get
\begin{multline} \label{EqTrucImmondeAuMilieu}
\OpeDiv(B_Au \nabla \delta_h \partial_\nu v) =   \delta_h \partial_\nu q  - \sum_{0 \leq \alpha \leq \nu} \binom{\nu}{\alpha} \OpeDiv\left[\left(\delta_h \partial_{\nu-\alpha} B_A u\right) \nabla \tau_h \partial_{\alpha} v\right].  \\
- \sum_{0 \leq \alpha < \nu} \binom{\nu}{\alpha} \OpeDiv\left[\left(\partial_{\nu-\alpha} B_A u\right) \nabla \delta_h \partial_{\alpha} v\right].
\end{multline}
Now, Lemma \ref{LemEUFirstOrder} tells us that this implies
 \begin{multline*}
   \left\| \delta_h \partial_\nu v \right\|_{H^1}
\leq 
 C_\epsilon \|\delta_h \partial_\nu q\|_{(H^1_\NotZeroMeanValue)^*}  \\
  + C_\epsilon\sum_{0 \leq \alpha \leq \nu} \binom{\nu}{\alpha} \left\|\OpeDiv\left[\left(\delta_h \partial_{\nu-\alpha} B_A u\right) \nabla \tau_h \partial_{\alpha} v\right]\right\|_{(H^1_\NotZeroMeanValue)^*} \\
+ C_\epsilon\sum_{0 \leq \alpha < \nu} \binom{\nu}{\alpha} \left\|\OpeDiv\left[\left(\partial_{\nu-\alpha} B_A u\right) \nabla \delta_h \partial_{\alpha} v\right]\right\|_{(H^1_\NotZeroMeanValue)^*} .
	 \end{multline*}
Since  $ \|\delta_h \partial_\nu q\|_{(H^1_\NotZeroMeanValue)^*} \leq  \|\partial_\nu q\|_{L^2}$, this bound is uniform in $h$, and so it is enough to ensure $v \in H^{n+1}$ and
 \begin{equation}\label{EqRevisedTameEstimate}
   \| v\|_{H^{n+1}} \leq  C\left\{\| q\|_{H^{n-1}} +  \sum_{0 \leq k \leq n-1} (1+\|u\|_{\SetCl^{n-k+2}}) \|v\|_{H^{k+1}}\right\}. 
\end{equation}
Notice that, when $n>1$, the following Landau--Kolmogorov inequalities hold
\begin{gather*}
\| u \|_{\SetCl^{n-k+2}} \leq C_{k,n} \|u \|_{\SetCl^3}^{1-\frac{k}{n-1}} \| u \|_{\SetCl^{n+2}}^{\frac{k}{n-1}}, \\
\| v \|_{H^{k+1}}\leq C_{k,n} \| v \|_{H^1}^{\frac{k}{n-1}} \| v \|_{H^{n}}^{1-\frac{k}{n-1}}.
\end{gather*}
They are quite classical and can be easily proved by induction from
\[\| \xi \|_{\SetCl^1} \leq \sqrt{2 \|\xi \|_{\SetCl^0} \|\xi\|_{\SetCl^2}} \qquad \text{and} \qquad \|\xi\|_{H^1} \leq \sqrt{\|\xi\|_{L^2} \|\xi\|_{H^2}},\]
for $\xi$ smooth enough satisfying $\int \xi = 0$. Since $a^{1-t}b^t \leq (1-t)a+tb$, we get
\[  \|u\|_{\SetCl^{n-k+2}} \|v\|_{H^{k+1}} \leq \frac{k}{n-1}\|u\|_{\SetCl^3}\|v\|_{H^{n}} +  \left(1-\frac{k}{n-1}\right)\|u\|_{\SetCl^{n+2}}\|v\|_{H^1}, \]
and therefore
\[ \|v\|_{H^{n+1}} \leq C \left\{ \|q\|_{H^{n-1}} + (1+\|u\|_{\SetCl^3})\|v\|_{H^{n}} + \|u\|_{\SetCl^{n+2}}\|v\|_{H^1} \right\}. \]
This last inequality still holds when $n=1$, thanks to \eqref{EqRevisedTameEstimate}. In any case,  as $\|v\|_{H^1} \leq C_\epsilon \|q\|_{H^{-1}}$ and $\|u\|_{\SetCl^{3}} + \|q\|_{H^{-1}} \leq M$, using our assumption \eqref{EqBoundednessInductionPreviousOrder}, 
\[  \| v\|_{H^{n+1}} \leq C_{\epsilon,M, n} \left\{ \|q\|_{H^{n-1}} + \| u \|_{\SetCl^{n+2}} \right\}.\]
This is exactly what we wanted.
\end{proof}

\begin{lemma}\label{LemFOBijection}
Under the same assumptions, for any $q \in \SetCl^{n,\alpha}_\NotZeroMeanValue(\SetT^N)$, there is a unique $v \in \SetCl^{n+2,\alpha}_\NotZeroMeanValue(\SetT^N)$ such that
\begin{equation*}
	D_u \FirstOperator(A,u)v = q.
\end{equation*}
\end{lemma}

\begin{proof}
If $q \in \SetCl^{n,\alpha}_\NotZeroMeanValue$, then $q \in H^n_\NotZeroMeanValue$, and thus according to the previous lemmas, there is $v \in H^{n+2}_\NotZeroMeanValue$ such that $D_u \FirstOperator(A,u) v = q$ in $(H^1_\NotZeroMeanValue)^*$. But since $\int q = 0$, given the particular form of $D_u \FirstOperator(A,u)v$ given by Lemma \ref{LemDerivativeFirstOperator}, such an equality in fact holds in $H^{-1}$. Thus, {locally}, in a weak sense,
\[ D_u \FirstOperator(A,u)v = q.\]
Then, we can locally use the theory of regularity for the solutions to a strictly elliptic equation in $\SetR^N$ to get existence and uniqueness of $v \in \SetCl^{n,\alpha}$ (cf. for instance Gilbarg \& Trudinger \cite{MR0473443}, Chapter 6).
\end{proof}

\begin{theorem}\label{ThePreRegularityKantorovich}
For any $A \in \SetSymmetric^{++}_N$, let $\OpeKantorovich{A}$ be the Kantorovich potential between the probability measure $\mu$ and $\nu$, which are still assumed to have smooth, strictly positive densities. Then, for any $n \geq 0$ and $\alpha \in (0,1)$, the following map
\[\Psi:\left\{ \begin{array}{ccc}
		\SetSymmetric_N^{++} &\longrightarrow &\SetCl^{n+2,\alpha}(\SetT^N) \\
		A &\longmapsto &\OpeKantorovich{A}
	\end{array}\right. ~\text{is}~~~ \SetCl^1.\] 
\end{theorem}

\begin{proof}
We denote by $\Omega$ be the set of all $(A,u) \in \SetSymmetric_{N}^{++} \times \SetCl^{n+2,\alpha}_\NotZeroMeanValue(\SetT^N)$ such that $A -D^2 u > 0$. Then $\Omega$ is open, the operator $\FirstOperator : \Omega \rightarrow  \SetCl^{n,\alpha}_\NotZeroMeanValue(\SetT^N)$, defined by
\[ \FirstOperator(A,u) = f - g\left(\idmap - A^{-1} \nabla u\right) \det\left(\idmatrix - A^{-1} D^2 u\right), \]
is smooth and, according Lemma \ref{LemFOBijection}, $D_u \FirstOperator(A,\psi_A) : \SetCl^{n+2,\alpha}_\NotZeroMeanValue(\SetT^N) \rightarrow  \SetCl^{n,\alpha}_\NotZeroMeanValue(\SetT^N)$ is a bijection. From the Banach--Schauder theorem, we deduce it is an isomorphism.
Since $\FirstOperator(A,\OpeKantorovich{A}) = 0$, according to the implicit function theorem, there is a $\SetCl^1$ map $\Phi$ defined in a neighborhood of $A$ such that 
 $B - D^2\Phi_{B} > 0$ and, for any $u \in \SetCl^{n+2,\alpha}_\NotZeroMeanValue$, $\FirstOperator(B,u)= 0$ if and only if $u=\Phi_B$. According to Lemma \ref{LemCharacterizationAPriori1}, it implies $\Phi_B = \OpeKantorovich{B}$. Thus, globally, $\Psi = \Phi$ is a $\SetCl^1$ map $\SetSymmetric_N^{++} \rightarrow\SetCl^{n,\alpha}_\NotZeroMeanValue(\SetT^N)$.
\end{proof}

We are now going to apply this result to the cost $c$ defined by \eqref{EqCostDefinition}, that is to say the cost induced by the matrix 
\[ A_t := \left( 	\begin{array}{ccccc}
				1	\\
					& \lambda_1(t) \\
					&			& \lambda_1(t) \lambda_2(t) \\
					&			& 					& \ddots \\
					&			&					&		& \prod \lambda_i(t) 
			\end{array}\right),\]
where  $\lambda_1, \ldots, \lambda_{N-1} : \SetR \rightarrow [0,+\infty)$ are assumed to be such that $\lambda_k(t) = 0$ if and only if $t=0$. 

\begin{theorem}\label{TheRegularityKantorovich}
If $\lambda_1, \ldots, \lambda_{N-1}$ are smooth, the map $\psi : t \mapsto \OpeKantorovich{A_t}$ is $\SetCl^1$, and satisfies:
\begin{equation}\label{EqRegularityKantorovich}
 \OpeDiv\left\{ f \left[ A_t - D^2 \psi_t \right]^{-1} \left( \nabla \dot{\psi}_t - \dot{A}_t A^{-1}_t \nabla \psi_t \right)\right\} = 0.
 \end{equation}
Moreover, if $u  : (0,+\infty) \rightarrow \SetCl^{n+2,\alpha}(\SetT^N)$ is $\SetCl^1$ and satisfies, for all $t \in (0,+\infty)$,
\begin{equation}\label{EqRegularityKantorovich2}
 A_t - D^2 u > 0 \quad \text{and} \quad \OpeDiv\left\{ f \left[ A_t - D^2 u_t \right]^{-1} \left( \nabla \dot{u}_t - \dot{A}_t A^{-1}_t \nabla u_t \right)\right\} = 0,
 \end{equation}
and if $u_{t_0} = \psi_{t_0}$ for some $t_0 > 0$, then $u_t = \psi_t$ for all $t > 0$.
\end{theorem}

\begin{proof}
If $\psi_t := \OpeKantorovich{A_t}$, for all $t > 0$, we have $\FirstOperator(A_t,\psi_t) = 0$. If we differentiate with respect to $t$, we get
\[ D_u \FirstOperator(A_t, \psi_t)\dot{\psi}_t + D_A \FirstOperator(A_t,\psi_t) \dot{A}_t = 0. \]
We have seen in Lemma \ref{LemDerivativeFirstOperator} that
\[ D_u \FirstOperator(A_t, \psi_t)\dot{\psi}_t  = \OpeDiv \left(f \left[A_t - D^2  \psi_t \right]^{-1} \nabla \dot{\psi}_t\right). \]
On the other hand,
\[ D_A \FirstOperator(A_t,\psi_t) \dot{A}_t = - \OpeDiv \left(f \left[A_t - D^2 \psi_t \right]^{-1} \dot{A}_t A^{-1} \nabla \psi_t \right).\]
We thus get \eqref{EqRegularityKantorovich}.

If $u  : (0,+\infty) \rightarrow \SetCl^{n+2,\alpha}(\SetT^N)$ is $\SetCl^1$ and satisfies \eqref{EqRegularityKantorovich2},  with $u_{t_0} = \psi_{t_0}$ for some $t_0 > 0$, then $\FirstOperator(A_t,u_t)$ must be constant and equal to $\FirstOperator(A_{T_0},u_{t_0}) = 0$. Thus, according to Lemma \ref{LemCharacterizationAPriori1}, $u_t = \Psi_{A_t}$.
\end{proof}

%
%

\section{Initial condition in dimension $2$}

Due to the very technical nature of the proofs, the following sections will only deal with the dimension $2$, to ease the exposition. In the last section, we shall explain what changes in higher dimension.

\medskip

Let $\lambda:\SetR \rightarrow [0,+\infty)$ be a smooth function such that $\lambda_t=0$ if and only if $t=0$.
From now on, we will only consider the cost induced by
\[ A_t = \left( 	\begin{array}{cc}
				1 & 0 \\
				0 & \lambda_t
			\end{array}\right).\]
For $t \neq 0$, let $\psi_t$ be the associated Kantorovich potential between the probability measures $\mu,\nu$, assuming they have the same properties as before (that is, strictly positive and smooth densities), and let $T_t$ be the corresponding optimal transport map. Then, according to Theorem \ref{TheRegularityKantorovich}, $t \mapsto \psi_t$ and $t \mapsto T_t$ are $\SetCl^1$ on $\SetR \setminus \{ 0 \}$. Moreover, Carlier, Galichon, and Santambrogio \cite{MR2607321} proved:

\begin{theorem}[C.--G.--S.]
	As $t \rightarrow 0$, the map $T_t$ converges to the Knothe--Rosenblatt rearrangement $R$ in $L^2(\SetT^2,\mu;\SetT^2)$.
\end{theorem}

Let us denote by $u^1_0(\cdot)$ and $u^2_0(x_1,\cdot)$ the Kantorovich potentials for respectively $R^1(\cdot)$ and $R^2(x_1,\cdot)$. Indeed, recall that $R^1$ sends the first part $\mu^1$ of the disintegration of $\mu$ onto the first part $\nu^1$ of the disintegration of $\nu$, and that $R^2(x_1,\cdot)$ sends the second part $\mu^2_{x_1}$ onto $\nu^2_{R^1(x_1)}$, in an optimal way for the squared distance on the $1$-dimensional torus $\SetT^1$; hence these transport maps come from some potentials. We have:
\[ R(x) = \left(\begin{array}{c}
					x_1-\partial_1 u^1_0(x_1) \\
					 x_2 - \partial_2 u^2_0(x_1,x_2)
				\end{array}\right).\]
The Carlier--Galichon--Santambrogio theorem suggests some connexion exists between $\psi_t$ and $(u^1_0,u^2_0)$. Since $T_t = \idmap - A_t^{-1}\nabla \psi_t$, let us follow our instinct and set
\[ \psi_t(x_1,x_2) = \psi_t^1(x_1) + \lambda_t \psi_t^2(x_1,x_2),\]
and, to ensure uniqueness, require 
\[ \int \psi^1_t(x_1) \OpeIntd x_1 = 0 \qquad \text{and} \qquad \int \psi^2_t(x_1,x_2) \OpeIntd x_2 = 0.\]
Notice that $\psi^1_t$ and $\psi^2_t$ are then uniquely determined, and are smooth, since 
\[ \psi^1_t(x_1) = \int \psi_t(x_1,x_2) \OpeIntd x_2 \qquad \text{and} \qquad \psi^2_t(x) = \frac{1}{\lambda_t}\left(\psi_t(x) - \psi^1_t(x_1)\right).\]
Let us denote by $E$ the set of all $(t,u^1,u^2) \in \SetR \times \SetCl^\infty(\SetT^1) \times \SetCl^\infty(\SetT^2)$ such that
\[ \int u^1(x_1) \OpeIntd x_1 = 0 \qquad \text{and} \qquad \int u^2(x_1,x_2) \OpeIntd x_2 = 0,\]
and by $\Omega$ the open subset of $E$ formed by the $(t,u^1,u^2)$ such that:
\begin{itemize}
\item either $t \neq 0$, and then $A_t - D^2(u^1 + \lambda_tu^2) > 0$;
\item or $t=0$, and then $1-\partial_{1,1} u^1 > 0$ and $1-\partial_{2,2} u^2 > 0$.
\end{itemize}
Then, thanks to Lemma \ref{LemCConcave}, we can define an operator $\SecondOperator : \Omega \rightarrow \SetCl^\infty(\SetT^2)$ by setting, when $t \neq 0$,
\begin{equation} \label{EqDefSecondOperator}
 \SecondOperator(t,u^1,u^2) := \FirstOperator(A_t,u^1 + \lambda_t u^2),
\end{equation}
where $\FirstOperator$ is the operator introduced in Section \ref{SecNonDegenerate}:
\[ \FirstOperator(A,u) = f - g\left(\idmap - A^{-1} \nabla u\right) \det\left(\idmatrix - A^{-1} D^2 u\right). \]
Since, according to Lemma \ref{LemCharacterizationAPriori1}, $\FirstOperator(A_t,u) = 0$ if and only if $u = \psi_t$, we have:

\begin{lemma}\label{LemCharacterizationAPriori2}
For any $(t,u^1,u^2) \in \Omega$, $\SecondOperator(t,u^1,u^2) = 0$ if and only if $u^1 = \psi_t^1$ and $u^2 = \psi^2_t$.
\end{lemma}

Now, we are going to extend $\SecondOperator$ for $t=0$.
Notice indeed that
\begin{multline*}
 A^{-1} \nabla (u^1 + \lambda_t u^2) = \left(
\begin{array}{c}
\partial_1 u^1 + \lambda_t \partial_1 u^2 \\
\partial_2 u^2
\end{array}
\right) \\
 \text{and} \qquad A^{-1} D^2 (u^1 + \lambda_t u^2) =  \left(
\begin{array}{cc}
\partial_{1,1} u^1 + \lambda_t \partial_{1,1} u^2 	& 	\lambda_t \partial_{1,2} u^2 \\
\partial_{1,2} u^2							& 	\partial_{2,2} u^2
\end{array}
\right),
\end{multline*}
therefore we can smoothly extend $\SecondOperator$. If we conveniently define an operator $\partial$ by setting $\partial u := (\partial_1 u^1, \partial_2 u^2)$, then $R = \idmap - \partial \psi_0$, and 
\begin{equation}\label{EqSecondOperatorZero}
 \SecondOperator(0, u^1, u^2) = f - g\left(\idmap - \partial u\right)\det\left( \idmatrix - D\partial u\right).
 \end{equation}
Alas, we cannot do the same as in the previous section and apply the implicit function theorem, for if we solve $D_u \mathcal{G}(0,\psi^1_0,\psi^2_0)(v^1,v^2) = q$, then \emph{a priori} the solution $v^2$ is not smooth enough. Indeed, as we will see later, if $q \in H^n$, then $v^1 \in H^{n+2}$, but we can only get $v^2 \in H^n$. However, we can circumvent this difficulty by considering $\SetCl^\infty$ functions, so as to have an infinite source of smoothness, and use the Nash--Moser implicit function theorem instead of the “classical” implicit function theorem.

\medskip

Before stating our next result, let us recall some definitions from the Nash--Moser theory. For more details, see for instance Hamilton \cite{MR656198}.

Let $X$ and $Y$ be two Fréchet spaces, endowed each one with a family of increasingly stronger semi-norms, denoted by  $\{\|\cdot\|_n^X\}_{n \geq 0}$ and $\{\|\cdot\|_n^Y\}_{n \geq 0}$. For instance, you can think of $\SetCl^\infty(\SetT^2)$, with the norms $\| \cdot \|_n = \| \cdot \|_{\SetCl^n}$ or equivalently   $\| \cdot \|_n = \| \cdot \|_{H^n}$. A map $\phi : U \rightarrow Y$ is said to be “tame” if it is defined on an open set $U \subset X$, is continuous, and in a neighborhood $V$ of each point, one can find $r \geq 0$, $b \geq 0$ and a sequence $(C_n)_{n \geq b}$ of positive constants such that the following “tame estimate” is satisfied:
\[ \forall x \in V, \forall n \geq b, \| \phi(x) \|_n^Y \leq C_n \left( 1 + \| x \|_{n+r}^X \right).\]
Notice that $r, b, C_n$ can depend on $V$, but $V$ cannot change with $n$. The map $\phi$ is “smooth tame” if it is smooth and if all its Gâteaux derivative $D^k \phi : U \times X \rightarrow Y$ are tame. From the definition \eqref{EqDefSecondOperator} of $\mathcal{G}$, we easily get:

\begin{lemma}
The operator $\mathcal{G} : \Omega \rightarrow \SetCl^\infty(\SetT^2)$ is smooth tame.
\end{lemma}

The Nash--Moser theorem holds for some Fréchet spaces, the so-called “tame spaces” defined as follows. If $E$ is a Banach space, the space of exponentially decreasing sequences in $E$ is defined as:
\[ \Sigma(E) := \left\{ (u_n) \in E^\SetN ~~; \quad \forall n \in \SetN, ~~ \| u \|_n^{\Sigma(E)} := \sum_{k=0}^\infty e^{nk} \|u_k\|^E < \infty \right\}\]
A Fréchet space $X$ is said to be “tame” is there is a Banach space $E$ and two tame linear maps $\Phi : X \rightarrow \Sigma(E)$ and $\Psi : \Sigma(E) \rightarrow X$ such that $\Psi \circ \Phi = \idmap_X$. For instance, $\SetCl^\infty(\SetT^2)$ is a tame space. If $X$ and $Y$ are tame, then so is their cartesian product $X \times Y$.

\begin{theorem}[Nash--Moser]
Let $X$ and $Y$ be two tame spaces. Let $U \subset X$ be an open subset and $\Phi : U \rightarrow Y$ be a smooth tame map. We assume that, for any $u \in U$ and any $q \in Y$, there is a unique $v \in X$ such that $D\Phi(u)v = q$. If the inverse operator $\InverseOperator : U \times Y \rightarrow X$ is a smooth tame map, then $\Phi$ is locally invertible, and the local inverse maps are smooth tame.
\end{theorem}

\begin{corollary}[implicit function]\label{CorNMImplicitFunction}
Let $X, Y, Z$ be three tame spaces, and $U \subset X$ and $V \subset Y$ be open subsets. We assume $\Phi : U \times V \rightarrow Z$ is a smooth tame map such that $\Phi(u_0,v_0) = 0$ for some $(u_0,v_0) \in U \times V$. If, for any $u \in U$, $v \in V$ and $q \in Z$, there is a unique $w \in Y$ such that $D_v \Phi(u,v)w = q$, and if the inverse operator $\InverseOperator : U \times V \times Z \rightarrow Y$ is a smooth tame map, then there is a smooth tame map $\psi$ defined in a neighborhood of $u_0$ and taking values in a neighborhood of $v_0$ such that $\Phi(u,v) = 0$ if and only if $v = \psi(u)$.
\end{corollary}

\medskip

Notice that we need only to use this last statement in an open neighborhood of $(0,u_0^1,u_0^2) \in \Omega$, with $u^1_0$ and $u^2_0$ the Kantorovich potential associated with the Knothe--Rosenblatt rearrangement. Let us define this neighborhood $\Omega_0$ in the following way: First, take $\epsilon > 0$ such that $1-\partial_{1,1} u^1_0 > \epsilon$ and $1-\partial_{2,2} u^2_0 > \epsilon$. Then we take for $\Omega_0$ the set of all $(t,u^1_t, u^2_t) \in \Omega$ such that:
\begin{equation}\label{EqDefOmegaZeroA}
\text{if } t=0, \quad 1-\partial_{1,1} u^1 > \epsilon \quad \text{and} \quad 1-\partial_{2,2} u^2 > \epsilon,
\end{equation}
and
\begin{equation}\label{EqDefOmegaZeroB}
\text{if } t \neq 0, \quad  1-\partial_{1,1} u^1 - \lambda_t \partial_{1,1} u^2 > \epsilon \quad \text{and} \quad A_t - D^2(u^1 + \lambda_t u^2) > \epsilon \lambda_t.
\end{equation}

Recall that we denote with a $\NotZeroMeanValue$ subscript the sets of maps with zero mean value: $\SetCl^\infty_\NotZeroMeanValue$ is thus the set formed by the smooth functions $u$ such that $\int u = 0$. When $X$ is a $2$-variable function space, we also denote by a “$\NotStarZeroMeanValue$”  subscript, as in $\SetCl^\infty_\NotStarZeroMeanValue(\SetT^2)$ the set formed by the $\xi \in X$ such that $\int \xi(\cdot,x_2) \OpeIntd x_2 = 0$.

\begin{theorem}\label{ThePreNashMoser}
For all $(t,u^1,u^2) \in \Omega_0$, for any $q \in \SetCl^\infty_\NotZeroMeanValue(\SetT^2)$, there is a unique $(v^1,v^2) \in \SetCl^\infty_\NotZeroMeanValue(\SetT^1) \times \SetCl^\infty_\NotStarZeroMeanValue(\SetT^2)$ such that
\begin{equation}\label{EqDerivative} 
	D_u \SecondOperator(t,u^1,u^2)(v^1,v^2) = q,
\end{equation}
Moreover, the inverse operator
\[ \InverseOperator : \left\{ 	\begin{array}{ccl}
					\Omega_0 \times \SetCl^\infty_\NotZeroMeanValue(\SetT^2)& \rightarrow& \SetCl^\infty_\NotZeroMeanValue(\SetT^1) \times \SetCl^\infty_\NotStarZeroMeanValue(\SetT^2) \\
					\left((t,u^1,u^2),q\right) & \mapsto &(v^1,v^2)
				\end{array} \right.\]
is smooth tame.
\end{theorem}

\begin{proof}
We will show the existence of $(v^1, v^2)$ in  Section \ref{SecInvertibility}.
We also report the proof of the existence of a tame estimate for the inverse operator $\InverseOperator$ to Section \ref{SecEstimate}.

Let us conclude from that point. Now all that remains to show is that $\InverseOperator$ is continuous, and that the derivative $D^k \InverseOperator$ are tame.

First, if $(t_k, u^1_k,u^2_k,q_k) \in \Omega_0$ converges towards $(t,u^1,u^2,q) \in \Omega_0$, for each $k$ let $(v^1_k,v^2_k)$ be the corresponding inverse. Thanks to the tame estimate from  Section \ref{SecEstimate}, 
 $v^1_k$ and $v^2_k$ are bounded in all the spaces $H^n$. Hence, compact embeddings provide convergence, up to an extraction, to some $v^1, v^2$ as strongly as we want, which, since $D\mathcal{G} $ is continuous, must be the solution of $D \mathcal{G}(t,u^1,u^2)(v^1,v^2) = q$.

Then, all the derivative $D^k \InverseOperator$ are also tame, since they give the solution to the same kind of equation as \eqref{EqDerivative}. Indeed, by differentiating  \eqref{EqDerivative}, we get 
\[ D_u \mathcal{G} D \InverseOperator = D q - D(D_u \mathcal{G}),\]
and then we can apply the results of  Section \ref{SecEstimate} once more.
\end{proof}

If we now set $\psi^1_0 = u^1_0$ and $\psi^2_0 = u^2_0$, with $u^1_0$ and $u^2_0$ the Kantorovich potentials for the Knothe--Rosenblatt rearrangement, we can state the following:

\begin{corollary}
The map $\displaystyle \left\{ \begin{array}{ccl}
				\SetR & \rightarrow & \SetCl^\infty_\NotZeroMeanValue(\SetT^1) \times \SetCl^\infty_\NotStarZeroMeanValue(\SetT^2) \\
				t &\mapsto& (\psi^1_t,\psi^2_t)
			\end{array}\right.$ is smooth.
\end{corollary}

\begin{proof}
On some interval $(-\tau, \tau)$, this is a direct consequence of Corollary \ref{CorNMImplicitFunction}, Theorem \ref{ThePreNashMoser}, and Lemma \ref{LemCharacterizationAPriori2}.  For larger $t$, it follows from Theorem \ref{ThePreRegularityKantorovich}.
\end{proof}

\begin{theorem}\label{TheCI}
The curve formed by the Kantorovich potentials $(\psi_t)$ is the only curve in $\SetCl^2_\NotZeroMeanValue(\SetT^2)$  defined on $\SetR$ such that, for $t \neq 0$,
\begin{equation}\label{EqKantorovich}
 A_t - D^2 \psi_t >0 ~~ \text{and} ~~ \OpeDiv\left( f \left[ A_t - D^2 \psi_t \right]^{-1} \left( \nabla \dot{\psi}_t - \dot{A}_t A^{-1}_t \nabla \psi_t \right)\right) = 0,
 \end{equation}
and that can be decomposed into two smooth curves $(\psi^1_t)$ and $(\psi^2_t)$ such that
\[ \psi_t(x_1,x_2) = \psi^1_t(x_1) + \lambda_t \psi^2_t(x_1,x_2),\]
 with $\psi^1_0$ and $\psi^2_0$ being the Kantorovich potentials for the Knothe rearrangement. 
\end{theorem}

\begin{proof}
Let $u_t = u^1_t + \lambda_tu^2_t$ be such a curve, and let us check that $u_t = \psi_t$. Since $u^1_0$ and $u^2_0$ are the potentials for the Knothe rearrangement,  $(0,u^1_0,u^2_0) \in \Omega_0$,  so  $(t,u^1_t,u^2_t)$ is in $\Omega_0$ at least for $t$ small. For $t \neq 0$, \eqref{EqKantorovich} is equivalent to
\[ D_u\mathcal{F}(t,u_t)\dot{u}_t + D_t \mathcal{F}(t,u_t) = 0,\]
and therefore
\[ D_u \mathcal{G}(t,u^1_t,u^2_t)(\dot{u}^1_t, \dot{u}^2_t)  + D_t \mathcal{G}(t,u^1_t, u^2_t) = 0.\]
By assumption, $\mathcal{G}(0,u^1_0,u^2_0) = 0$. Integrating in time, we get $\mathcal{G}(t,u^1,u^2) = 0$. Therefore, according to Lemma \ref{LemCharacterizationAPriori2}, $u^1_t = \psi^1_t$ and $u^2_t = \psi^2_t$, i.e. $u_t = \psi_t$.

For larger $t$, we apply Theorem \ref{TheRegularityKantorovich}.
\end{proof}

%
%

\section{Proof of the invertibility}\label{SecInvertibility}

We recall that
\[ \FirstOperator(A,u) = f - g\left(\idmap - A^{-1} \nabla u\right) \det\left(\idmatrix - A^{-1} D^2 u\right), \]
and
\begin{equation}\label{EqRappelSecondOperator}
 \SecondOperator(t,u^1,u^2) := \FirstOperator(A_t,u^1 + \lambda_t u^2) \qquad \text{with} \qquad A_t := \left( 	\begin{array}{cc}
				1 & 0 \\
				0 & \lambda_t
			\end{array}\right).
\end{equation}
We want to prove the invertibility of $D_u \SecondOperator(t,u^1,u^2)$. The first lemma will consider the case $t \neq 0$, the second one the case $t=0$.

\begin{lemma}\label{LemInversibilityZero}
For any $(t,u^1,u^2) \in \Omega_0$ with $t \neq 0$, for any $q \in \SetCl^\infty_\NotZeroMeanValue(\SetT^2)$, there is a unique   $(v^1,v^2) \in \SetCl^\infty_\NotZeroMeanValue(\SetT^1) \times \SetCl^\infty_\NotStarZeroMeanValue(\SetT^2)$ such that
\begin{equation}\label{EqDerivative2} 
	D_u \SecondOperator(t,u^1,u^2)(v^1,v^2) = q.
\end{equation}
\end{lemma}

\begin{proof}
Let $(t,u^1,u^2) \in \Omega_0$ with $t \neq 0$, and let $q \in \SetCl^\infty_\NotZeroMeanValue(\SetT^2)$.
Then, if we set $u_t := u^1+ \lambda_t u^2$, Lemma \ref{LemEUNextOrders}  tells us that there is a unique $v_t \in \SetCl^\infty_\NotZeroMeanValue(\SetT^2)$ such that
\begin{equation}\label{EqElliptic}
  \OpeDiv \left(\left(f-\SecondOperator(t,u^1,u^2)\right) \left[\idmatrix -A_t^{-1} D^2 u_t\right]^{-1} A_t^{-1} \nabla v_t\right) = q. 
\end{equation}
Let us define
\[ v^1(x_1) := \int v_t(x_1,x_2) \OpeIntd x_2 \quad \text{and} \quad v^2(x_1,x_2) := \frac{1}{\lambda_t}\left( v_t(x_1,x_2) - v^1(x_1)\right).\]
Then, by construction, $(v^1,v^2)$ is the unique pair solving \eqref{EqDerivative2}.
\end{proof}

\begin{lemma}\label{LemInversibilityNonZero}
For any $(0,u^1,u^2) \in \Omega_0$, for any $q \in \SetCl^\infty_\NotZeroMeanValue(\SetT^2)$, there is a unique   $(v^1,v^2) \in \SetCl^\infty_\NotZeroMeanValue(\SetT^1) \times \SetCl^\infty_\NotStarZeroMeanValue(\SetT^2)$ such that
\begin{equation*}
	D_u \SecondOperator(0,u^1,u^2)(v^1,v^2) = q.
\end{equation*}
\end{lemma}

\begin{proof}
We want to solve
\[
D_u \SecondOperator(0,u^1,u^2)(v^1,v^2) =   q. 
\]
By substitution, for any $\xi \in \SetCl^\infty$, Equation \eqref{EqSecondOperatorZero} yields
\[ \int \xi\left(x-\partial u(x)\right)\left[f(x) - \SecondOperator(0,u^1,u^2)(x)\right] \OpeIntd x = \int \xi(y) g(y) \OpeIntd y.\]
Therefore, if we differentiate the previous equation with respect to $u$ along the direction $v$, and recall our notation $\partial u = (\partial_1 u^1, \partial_2 u^2)$ and $\partial v = (\partial_1 v^1, \partial_2 v^2)$, we get
\begin{multline*}
- \int \left\langle \nabla \xi(\idmap -\partial u), \partial  v \right\rangle \left(f-\SecondOperator(0,u^1,u^2)\right) \\
	- \int \xi(\idmap - \partial u) D_u \SecondOperator(0,u^1,u^2) (v^1,v^2) = 0.
\end{multline*}
Since $\nabla [\xi \circ (\idmap - \partial u)] = \OpeTranspose{[\idmatrix - D\partial u]} \nabla \xi(\idmap - \partial u)$, we have
\begin{align*}  
	\left\langle \nabla \xi(\idmap - \partial u), \partial  v \right\rangle 	& 	= \left\langle \nabla [\xi \circ (\idmap - \partial u)], [\idmatrix - D\partial u]^{-1} \partial v \right\rangle 
\end{align*} 
and this yields
\[
D_u \SecondOperator(0,u^1,u^2)(v^1,v^2) =   \OpeDiv \left(\left(f-\SecondOperator(0,u^1,u^2)\right) \left[\idmatrix - D\partial u\right]^{-1} \partial v\right). 
\]
Notice then that
\[ \left(f-\SecondOperator(0,u^1,u^2)\right) \left[\idmatrix - D\partial u\right]^{-1} = g\left(\idmap - \partial u\right)\left(
\begin{array}{cc}
1 - \partial_{2,2} u^2 		&	0				\\	
\partial_{1,2} u^2		& 1-\partial_{1,1} u^1	\\
\end{array}
\right), \] 
thus,
\begin{multline*}
 D_u \SecondOperator(0,u^1,u^2)(v^1,v^2) \\
 = \partial_1 \left[ g\left(x-\partial u(x)\right)\left( 1- \partial_{2,2} u^2(x) \right) \partial_1 v^1(x_1) \right]  + \partial_2 \left[ \ldots \right].
 \end{multline*}
Therefore, if $D_u \SecondOperator(0,u^1,u^2)(v^1,v^2)  = q$, integrating with respect to $x_2$ yields
\begin{equation*}
\int  \partial_1 \left[ g\left(x-\partial u(x)\right)\left( 1- \partial_{2,2} u^2(x) \right)  \partial_1 v^1(x_1) \right]   \OpeIntd x_2= \int q(x) \OpeIntd x_2,
 \end{equation*}
 that is to say
 \begin{equation}\label{EqPreIntegrating}
  \partial_1 \left[ \left\{ \int g\left(x-\partial u(x)\right)\left( 1- \partial_{2,2} u^2(x) \right)  \OpeIntd x_2\right\} \partial_1 v^1(x_1) \right] = \int q(x) \OpeIntd x_2.
 \end{equation}
But there is a smooth map $Q : \SetT^1 \rightarrow \SetR$ such that $\partial_1 Q(x_1) = \int q(x_1,x_2) \OpeIntd x_2$, since 
 $\int q(x) \OpeIntd x = 0$, 
and it is unique if we require $Q(0) = 0$.
Thus, taking a primitive of \eqref{EqPreIntegrating}, there is a $c \in \SetR$ such that:
\[  \underbrace{\left[ \int g\left(x-\partial u (x)\right) \left(1-\partial_{2,2} u^2(x) \right) \OpeIntd x_2\right]}_{G(x_1)} \partial_1 v^1(x_1)  = Q(x_1) + c.\]
Since $G(x_1) > 0$, we get
\[ \partial_1 v^1 = \frac{Q+c}{G},\]
and this yields the unique possible value for $c$ since the integral w.r.t. $x_1$ of the right hand side must be zero. Combined with the condition $\int v^1 \OpeIntd x_1 = 0$, we thus have completely  characterized $v^1$.

Now, let us do the same for $v^2$. We have to solve the equation
\begin{multline*} 
\partial_2 \left[  g\left(\idmap-\partial u\right) \left(1-\partial_{1,1} u^1\right) \partial_2 v^2 \right] \\
 = q -  \partial_1 \left[ g\left(\idmap-\partial u\right)\left( 1- \partial_{2,2} u^2 \right) \partial_1 v^1\right] - \partial_2 \left[ g\left(\idmap-\partial u\right) \partial_{1,2} u^2 \partial_1 v^1 \right],
\end{multline*}
and this is exactly the same kind of equation as \eqref{EqPreIntegrating}. If we fix $x_1 \in \SetT^1$, the same reasoning can be applied here, and thus we get $v^2$.
\end{proof}

This ends the proof of the invertibility. All that is left to show is that we have some tame estimates.

%
%

\section{Proof of the tame estimates}\label{SecEstimate}

Our aim here is to show that, locally on $(t, u^1, u^2) \in \Omega_0$ and $q \in \SetCl^\infty_\NotZeroMeanValue(\SetT^2)$, for any $n \in \SetN$, there is a constant $C_n > 0$ such that, if
\begin{equation}\label{EqDerivative3} 
	D_u \SecondOperator(t,u^1,u^2)(v^1,v^2) = q
\end{equation}
for some $(v^1,v^2) \in \SetCl^\infty_\NotZeroMeanValue(\SetT^1) \times \SetCl^\infty_\NotStarZeroMeanValue(\SetT^2)$, then 
\[ \|v^1\|_{H^{n+2}} + \| v^2 \|_{H^n} \leq C_n\left( 1 + |t| + \|u^1\|_{H^{n+3}} + \|u^2\|_{H^{n+3}} + \|q\|_{H^n} \right).\]
In fact, we will prove something slightly stronger:
\begin{equation}\label{EqGoal}
 \|v^1\|_{H^{n+2}} + \| \partial_2 v^2 \|_{H^n} \leq C_n\left( \|u^1\|_{\SetCl^{n+3}} + \|u^2\|_{\SetCl^{n+3}} + \|q\|_{H^n} \right).
 \end{equation}
Indeed, since $\int v^2(x_1,x_2) \OpeIntd x_2 = 0$, we have a Poincaré inequality, which implies $\| v^2 \|_{H^n} \leq c_n\| \partial_2 v^2 \|_{H^n}$. Notice also that \eqref{EqGoal} would by itself prove there is uniqueness.

We start with the case $t \neq 0$. As the bound for $\| v^1 \|_{H^{n+2}}$ simply follows from Lemma \ref{LemEUNextOrders} and an integration with respect to $x_2$, we just have to find a bound for $ \| \partial_2 v^2 \|_{H^n} $. Let us begin with  $ \| \partial_2 v^2 \|_{L^2} $.

\begin{lemma}\label{LemPrePreEstimate}
Let $M, \epsilon > 0$. There are $C=C(M,\epsilon)$ such that, if $(t,u^1,u^2) \in \Omega_0$ and $q \in \SetCl^\infty_\NotZeroMeanValue(\SetT^2)$ satisfy
\begin{equation}\label{EqPrePreEstimateAssumption}
\|q\|_{L^2} + \|u^1\|_{\SetCl^3} + \| u^2\|_{\SetCl^3} \leq M,
 \end{equation}
if $(v^1,v^2)  \in \SetCl^\infty_\NotZeroMeanValue(\SetT^1) \times \SetCl^\infty_\NotZeroMeanValue(\SetT^2)$ is a solution of \eqref{EqDerivative3}, then
\begin{equation}\label{EqPrePreTame2} 
	\|\partial_2 v^2\|_{L^2} \leq  C.
\end{equation}
\end{lemma}

\begin{proof}
We set $u_t := u^1 + \lambda_t u^2$ and also $v_t := v^1 + \lambda_t v^2$.  Then, by assumption,
\[ D_u \FirstOperator(A_t, u_t)v_t = q.\]
The property \eqref{EqDefOmegaZeroB} in the definition of $\Omega_0$ ensures we an apply  Lemma \ref{LemEUNextOrders} and get
\begin{equation}\label{EqBoundaryOrderN}
 \| v_t\|_{H^{2}} \leq C_{\epsilon,N,1} \left\{ \|q\|_{L^2} + \| u_t \|_{\SetCl^{3}} \right\} \leq C.
 \end{equation}
We now set
\begin{align*}
 B_t 		& := \left(f-\SecondOperator(t,u^1,u^2)\right)\left[ \idmatrix - A_t^{-1} D^2 u_t\right]^{-1} A_t^{-1} \\
 		& = \frac{f-\SecondOperator(t,u^1,u^2)}{\det(A_t - D^2 u_t)} \OpeTranspose{\left[\OpeComatrix{(A_t - D^2 u_t)}\right]} \\
		& = \frac{g(\idmap - A_t^{-1}\nabla u_t)}{\det A_t} \OpeTranspose{\left[\OpeComatrix{(A_t - D^2 u_t)}\right]} \\
 \end{align*}
 so that, according to \eqref{EqRappelSecondOperator} and Lemma \ref{LemDerivativeFirstOperator},   Equation \eqref{EqDerivative3} becomes
 \begin{equation*}
  \OpeDiv(B_t\nabla v_t) = q.
  \end{equation*}
  Notice that $\det A_t = \lambda_t$ and
 \[
 \OpeComatrix{(A_t - D^2 u_t)}
 	 = \left(		\begin{array}{cc}
  					 \lambda_t- \lambda_t \partial_{2,2} u^2	& \lambda_t \partial_{1,2} u^2_t 	\\
					\lambda_t \partial_{1,2} u^2_t		& 1-\partial_{1,1} u^1_t 
  				\end{array}\right),
  \]
 therefore we can write 
 \begin{equation}\label{EqRevisedDecompositionMatrixField}
 B_t = U_t + V_t/\lambda_t
 \end{equation}
 with 
\begin{gather}
U_t :=  g(\idmap - A_t^{-1}\nabla u_t)  \left( 	\begin{array}{cc}
							1-\partial_{2,2} u^2 	& \partial_{1,2} u^2 \\
							\partial_{1,2} u^2		& 0
					\end{array}	\right), \label{EqDefU} \\
 V_t :=  g(\idmap - A_t^{-1}\nabla u_t)  \left( 	\begin{array}{cc}
							0	& 0 \\
							0	& 1-\partial_{1,1} u_t 
					\end{array}	\right). \label{EqDefV} 
\end{gather}
Thus, 
\[   q=\OpeDiv(B_t\nabla v_t) = \OpeDiv(U_t\nabla v_t) + \frac{1}{\lambda_t}\OpeDiv(V_t \nabla v_t). \]
As $\partial_2 v^1 = 0$,  we have $V_t \nabla v^1 = 0$. Since $v_t = v^1 + \lambda_t v^2$, we get
\begin{equation*}
\OpeDiv(U_t\nabla v_t)  + \OpeDiv(V_t \nabla v^2) = q,
 \end{equation*}
that is to say
\begin{equation}\label{EqSimple}
  \partial_2\left[g(\idmap - A_t^{-1}\nabla u_t)(1-\partial_{1,1} u_t ) \partial_2 v^2\right] = q-\OpeDiv(U_t\nabla v_t).
\end{equation}
Since $g > \delta$ for some $\delta$, and as \eqref{EqDefOmegaZeroB} in the definition of $\Omega_0$ means $1-\partial_{1,1} u_t > \epsilon$, allowing the constant $C$ to change from line to line we get
\begin{align*}
\|\partial_2 v^2\|^2_{L^2} 	&	\leq \frac{C}{\delta\epsilon} \int g(\idmap - A^{-1}\nabla u_t)(1-\partial_{1,1} u_t )|\partial_2 v^2|^2 \\
					& 	\leq C\int \left[q-\OpeDiv(U_t\nabla v_t)\right] v^2 \\
					& 	\leq C\left(\|q\|_{L^2} + \|U_t \nabla v_t\|_{H^1}\right) \|v^2\|_{L^2}
\end{align*}
However, since $\int v^2(x_1,x_2) \OpeIntd x_2 = 0$, we have $\| v^2 \|_{L^2} \leq C\|\partial_2 v^2 \|_{L^2}$.
Therefore,
\[ \| \partial_2 v^2\|_{L^2} \| v^2 \|_{L^2} \leq C\|\partial_2 v^2\|^2_{L^2} \leq C \left(\|q\|_{L^2} + \|U_t \nabla v_t\|_{H^1}\right) \|v^2\|_{L^2}.\]
Thus, since $\|U_t\|_{\SetCl^1} \leq C(1+\|u^1\|_{\SetCl^3} + \|u^2\|_{\SetCl^3}) \leq C$ as we can see from \eqref{EqDefU},
\[ \|\partial_2 v^2\|_{L^2} \leq C\left\{\|q\|_{L^2} + \|v_t\|_{H^2}\right\}.\]
Then, using \eqref{EqBoundaryOrderN}, we get the result.
 \end{proof}

We now proceed by induction to get an estimate for any order $n \in \SetN$.

\begin{lemma}\label{LemPreEstimate}
Under the same assumptions than in the previous lemma, for any $n \in \SetN$, there is a constant $C_n=C_n(M,\epsilon)$ such that
\begin{equation}\label{EqPreTame2} 
	 \| \partial_2 v^2 \|_{H^{n}} 
		\leq C_n 
			\left(  \| q\|_{H^{n}} + \| u^1 \|_{\SetCl^{n+3}} + \| u^2 \|_{\SetCl^{n+3}} \right).
\end{equation}
\end{lemma}

\begin{proof}
Let us assume \eqref{EqPreTame2} has been proved for some $n \in \SetN$, and let us show it holds even for $n+1$. Let $\nu \in \SetN^2$ be such that $|\nu|:=\nu_1+\nu_2=n+1$. Recall \eqref{EqSimple}, that is to say
\[\partial_2\left[g(\idmap - A^{-1}\nabla u_t)(1-\partial_{1,1} u_t ) \partial_2 v^2\right] = q-\OpeDiv(U_t\nabla v_t).\]
We already know from Lemma \ref{LemFOBijection} that $v_t = v^1 + \lambda_t v^2$ is smooth, therefore,
if we apply $\partial_\nu$, we get
\begin{multline*} 
\partial_2\left[g(\idmap - A^{-1}\nabla u_t)(1-\partial_{1,1} u_t )\partial_2\partial_{\nu} v^2\right]  \\
= - \sum_{0 \leq \alpha < \nu} \binom{\nu}{\alpha} \partial_2\left[\partial_{\nu-\alpha}\left\{ g(\idmap - A^{-1}\nabla u_t)(1-\partial_{1,1} u_t )\right\} \partial_2\partial_{\alpha} v^2\right] \\
+ \partial_\nu q - \partial_\nu\OpeDiv(U_t\nabla v_t).
\end{multline*}
On the other hand,  since $g > \delta$ and $1- \partial_{1,1} u_t > \epsilon$, we have
\begin{align*}
\|\partial_2 \partial_\nu v^2\|^2_{L^2} 	&	\leq \frac{1}{\delta\epsilon} \int g(\idmap - A^{-1}\nabla u_t)(1-\partial_{1,1} u_t )|\partial_2 \partial_\nu v^2|^2 \\
						& 	\leq -\frac{1}{\delta\epsilon} \int \partial_2 \left[ g(\idmap - A^{-1}\nabla u_t)(1-\partial_{1,1} u_t )\partial_2 \partial_\nu v^2\right] \partial_\nu v^2.
\end{align*}
Thus,
\begin{multline*}
\|\partial_2 \partial_\nu v^2\|^2_{L^2}  \\
	\leq \sum_{0 \leq \alpha < \nu}\binom{\nu}{\alpha} \int \left[\partial_{\nu-\alpha}\left\{ g(\idmap - A^{-1}\nabla u_t)(1-\partial_{1,1} u_t )\right\} \partial_2\partial_{\alpha} v^2\right] \partial_2 \partial_\nu v^2 \\
	-\frac{1}{\delta\epsilon}\int \left[ \partial_\nu q - \partial_\nu\OpeDiv(U_t\nabla v_t) \right] \partial_\nu v^2,
\end{multline*}
and therefore
\begin{multline*}
\|\partial_2 \partial_\nu v^2\|^2_{L^2}  \\
	\leq \sum_{0 \leq \alpha < \nu}C \left\|\partial_{\nu-\alpha}\left\{ g(\idmap - A^{-1}\nabla u_t)(1-\partial_{1,1} u_t )\right\} \partial_2\partial_{\alpha} v^2\right\|_{L^2} \left\|\partial_2 \partial_\nu v^2\right\|_{L^2} \\
	+C\left\| \partial_\nu q - \partial_\nu\OpeDiv(U_t\nabla v_t) \right\|_{L^2} \left\| \partial_\nu v^2\right\|_{L^2}.
\end{multline*}
As $\|\partial_\nu v^2\|_{L^2} \leq c\| \partial_2 \partial_\nu v^2\|_{L^2}$, we get
\begin{multline}\label{EqBigEquation}
\| \partial_2 \partial_\nu v^2 \|_{L^2} \leq C \sum_{0 \leq k \leq n} \left\| g(\idmap - A^{-1}\nabla u_t)(1-\partial_{1,1} u_t )\right\|_{\SetCl^{n+1-k}} \left\|\partial_2 v^2 \right\|_{H^k} \\
 +C\left\{ \left\| q \right\|_{H^{n+1}} +\| U_t\nabla v_t\|_{H^{n+2}}\right\}.
\end{multline}
On the one hand, we can use the same Landau--Kolmogorov inequalities as in the proof of Lemma \ref{LemEUNextOrders}, and use again the fact that $a^{1-t}b^t \leq (1-t)a + tb$, to get, for $0 \leq k \leq n$, the following bound:
\begin{multline*}
\left\| g(\idmap - A^{-1}\nabla u_t)(1-\partial_{1,1} u_t )\right\|_{\SetCl^{n+1-k}} \left\|\partial_2 v^2 \right\|_{H^k} \\
\leq c_n \left(\left\|g(\idmap - A^{-1}\nabla u_t)(1-\partial_{1,1} u_t )\right\|_{\SetCl^{n+1}} \left\|\partial_2 v^2 \right\|_{L^2}\right. \\
+\left.\left\| g(\idmap - A^{-1}\nabla u_t)(1-\partial_{1,1} u_t )\right\|_{\SetCl^{1}} \left\|\partial_2 v^2 \right\|_{H^n} \right).
\end{multline*}
Recall we have assumed \eqref{EqPreTame2} holds true for $n$, therefore, using \eqref{EqPrePreEstimateAssumption}, we get
\begin{multline}\label{EqBigEquationPar1}
\left\| g(\idmap - A^{-1}\nabla u_t)(1-\partial_{1,1} u_t )\right\|_{\SetCl^{n+1-k}} \left\|\partial_2 v^2 \right\|_{H^k} \\
\leq c_n \left(1 + \|q\|_{H^n} + \|u^1\|_{\SetCl^{n+3}} + \|u^2\|_{\SetCl^{n+3}}\right).
\end{multline}
On the other hand,
\begin{align*}
\|U_t\nabla v_t\|_{H^{n+2}}& = \| D^{n+1} (U_t\nabla v_t)\|_{H^1}\\
	& \leq C\left\{ \| U_t \|_{\SetCl^{n+2}}\| \nabla v_t\|_{H^1} + \| U_t \|_{\SetCl^1} \| \nabla v_t \|_{H^{n+2}}\right\} ,
\end{align*}
which, since $\|u^1\|_{\SetCl^3} + \|u^2\|_{\SetCl^3} \leq M$, implies
\[
 \|U_t\nabla v_t\|_{H^{n+2}}  \leq C\left\{ \left(1+ \| u^1 \|_{\SetCl^{n+4}} +  \| u^2 \|_{\SetCl^{n+4}}\right)\| v_t\|_{H^2} +\| v_t\|_{H^{n+2}}\right\} .
\]
Then, using Lemma \ref{LemEUNextOrders} we get
\begin{equation}\label{EqBigEquationPart2}
 \|U_t\nabla v_t\|_{H^{n+2}} \leq c_n \left(\|q\|_{H^{n}} + \|u^1 \|_{\SetCl^{n+4}} + \|u^2 \|_{\SetCl^{n+4}}\right). 
 \end{equation}
 Bringing together \eqref{EqBigEquation}, \eqref{EqBigEquationPar1}, and \eqref{EqBigEquationPart2}, we get the estimate we sought.
\end{proof}

\begin{lemma}\label{LemPreEstimateZero}
The result of Lemma \ref{LemPreEstimate} still stands for $t=0$, with the same constants.
\end{lemma}

\begin{proof}
Let $(0,u^1,u^2) \in \Omega_0$  and $q \in \SetCl^\infty_\NotZeroMeanValue(\SetT^2)$ such that
\begin{equation}\label{EqPrePreEstimateAssumption2}
 \|q\|_{L^2} + \|u^1\|_{\SetCl^3} + \| u^2\|_{\SetCl^3} \leq M,
 \end{equation}
Then, since $(s,u^1,u^2) \in \Omega_0$ for $s$ small enough,  we can proceed by approximation. Indeed, if $(v^1_s,v^2_s)$ is the solution to
\[ D_u \SecondOperator(s,u^1,u^2)(v^1_s,v^2_s) = q, \]
where $u^1, u^2, q$ have been all fixed, then all the $H^n$ norms of $v^1_s, v^2_s$ are bounded according to Lemma \ref{LemPreEstimate}. Up to an extraction, there is convergence, which by compact embedding is as strong as we want. But the convergence can only be towards the solution of
\[ D_u \SecondOperator(0,u^1,u^2)(v^1,v^2) = q, \]
hence estimate \eqref{EqPreTame2} is still valid for the limit.
\end{proof}

This proves the existence of tame estimates, at least in dimension $2$. Let us now see what changes in higher dimension.

%
%

\section{Higher dimension}

The difficulty in extending those results in higher dimension only comes from the technical nature of Sections \ref{SecInvertibility} and \ref{SecEstimate}. 
We need a decomposition, not only of the potential, but also of the field matrix $B$, extending \eqref{EqRevisedDecompositionMatrixField}. The existence of such a decomposition is the only new difficulty.

\subsection*{Setting and notations}

We consider $\lambda_1, \ldots, \lambda_{N-1} : \SetR \rightarrow [0,+\infty)$, assumed to be smooth and such that $\lambda_k = 0$ if and only if $t=0$. We then define $A_t$ by
\[ A_t := \left( 	\begin{array}{ccccc}
				1	\\
					& \lambda_1(t) \\
					&			& \lambda_1(t) \lambda_2(t) \\
					&			& 					& \ddots \\
					&			&					&		& \prod \lambda_i(t) 
			\end{array}\right).\]
The decomposition of the Kantorovich potential $\psi_t$ becomes
\[ \psi_t(x_1,\ldots,x_N) = \psi^1_t(x_1) + \lambda_1 \psi^2_t(x_1,x_2) + \ldots + \left(\prod_{i < N} \lambda_i\right) \psi^N_t(x_1,\ldots,x_N).\]
where $\psi_t^k$ depends only on the $k$ first variables $x_1,\ldots,x_k$, and is such that
\[ \forall x_1, \ldots, x_{k-1}, \qquad \int \psi_t^k(x_1,\ldots,x_{k-1},y_k) \OpeIntd y_k = 0.\]
For convenience, we set
\[ \NotDecomPsi^k_t := \psi_t^k + \lambda_k \psi_t^{k+1} + \ldots + \left( \prod_{k \leq i < N} \lambda_i \right) \psi^N_t,\]
so that we have
\[ \NotDecomPsi_t^1 := \psi_t, \qquad \NotDecomPsi^k_t = \psi_t^k + \lambda_k \NotDecomPsi^{k+1}_t , \qquad \NotDecomPsi^N_t = \psi_t^N,\]
and
\[  \forall x_1, \ldots, x_{k-1},  \qquad \idotsint \NotDecomPsi^k_t(x_1,\ldots,x_{k-1}, y_k, \ldots, y_N) \OpeIntd y_k \ldots \OpeIntd y_N = 0.\]
For instance, if $d=3$, we have
\[ \psi_t = \psi^1_t + \lambda_1 \psi^2_t + \lambda_1 \lambda_2 \psi^3_t \qquad \text{and} \qquad 
	\left\{ 	\begin{array}{l}
				\NotDecomPsi_t^1 = \psi^1_t + \lambda_1 \psi^2_t + \lambda_1 \lambda_2 \psi^3_t \\
				\NotDecomPsi_t^2 =  \psi^2_t + \lambda_2 \psi^3_t \\
				\NotDecomPsi_t^3 = \psi_t^3.
			\end{array}\right.\]
Let us denote by $E$ the set of all $(t,u^1,\ldots,u^N) \in \SetR \times \prod \SetCl^\infty(\SetT^k)$ such that
\[ \forall k \in \{1,\ldots, N\}, \quad \int u^k \OpeIntd x_k = 0.\]
Then, if $(t,u^1,\ldots,u^N) \in E$, we set 
\[ \NotGathered{u}^N := u^N, \qquad \NotGathered{u}^k := u^k + \lambda_k \NotGathered{u}^{k+1}, \qquad u := \NotGathered{u}^1,\]
and this is consistent with the previous notation. Notice that
\[ \nabla u =
	\left(	\begin{array}{c}
			\partial_1 \NotGathered{u}^1 \\
			\lambda_1 \partial_2 \NotGathered{u}^2 \\
			\lambda_1 \lambda_2 \partial_3 \NotGathered{u}^3 \\
			\vdots \\
			\prod \lambda_k \partial_N \NotGathered{u}^N
		\end{array}\right)
 	\qquad \text{and} \qquad A^{-1} \nabla u =  \partial \NotGathered{u} = 
	\left(	\begin{array}{c}
			\partial_1 \NotGathered{u}^1 \\
			\partial_2 \NotGathered{u}^2 \\
			\partial_3 \NotGathered{u}^3 \\
			\vdots \\
			\partial_N \NotGathered{u}^N
		\end{array}\right),\]
and thus,
\begin{equation} \label{EqHessian}
	 A^{-1} D^2 u = D\partial \NotGathered{u} = \left(	\begin{array}{cccc}
			\partial_{1,2} \NotGathered{u}^1 	& 0 & \cdots &0 \\
			\partial_{1,2} \NotGathered{u}^2  	& \partial_{2,2} \NotGathered{u}^2 	& \ddots & \vdots \\
			\vdots 						& \vdots						& \ddots 	& 0 \\
			\partial_{1,N} \NotGathered{u}^N	& \partial_{2,N} \NotGathered{u}^N & \cdots 	& \partial_{N,N} \NotGathered{u}^N
		\end{array}\right). 
\end{equation}
We define $\Omega$ as the open subset of $E$ formed by the $(t,u)$ such that:
\begin{itemize}
\item either $t \neq 0$, and then $A_t - D^2 u > 0$;
\item or $t=0$, and then $1-\partial_{k,k} u^k > 0$ for all $k$.
\end{itemize}
As previously, we need only to work on a neighborhood $\Omega_0$ of the Kantorovich potential $(0, u^1_0, u^2_0)$, which we will define precisely later.

\subsection*{Invertibility}

We want to solve, for $(0,u) \in \Omega_0$, the equation $D_u \SecondOperator(0,u)v = q$. Since for $t > 0$,
\[   D_u \SecondOperator(t,u)v  = \OpeDiv \left(\left(f-\SecondOperator(t,u)\right) \left[\idmatrix -A^{-1} D^2 u\right]^{-1} A^{-1} \nabla v\right), \]
which, when replacing $A^{-1} D^2 u$ and $A^{-1} \nabla v$ with $D\partial \NotGathered{u}$ and $\partial \NotGathered{v}$, becomes
\[   D_u \SecondOperator(t,u)v = \OpeDiv \left(\left(f-\SecondOperator(t,u)\right) \left[\idmatrix - D\partial \NotGathered{u}\right]^{-1} \partial \NotGathered{v}\right) , \]
and since, when $t=0$,  we have $\NotGathered{u}^k = u^k$ and $\partial \NotGathered{u} = \partial u$, what we would like to solve is 
 \[  q = D_u \SecondOperator(0,u)v = \OpeDiv \left(\left(f-\SecondOperator(0,u)\right) \left[\idmatrix - D\partial {u}\right]^{-1} \partial {v}\right). \]
The trick is to integrate with respect to $x_{k+1},\ldots,x_N$ to get an equation on $v^1,\ldots, v^k$. If $v^1, \ldots, v^{k-1}$ have already been found, $[\idmatrix - D\partial u]^{-1}$ being lower triangular thanks to \eqref{EqHessian}, the resulting equation on $v^k$ is of the same kind as the one we have dealt with in Section \ref{SecInvertibility}. The same reasoning can thus be applied.

\subsection*{Tame estimate}

As in the $2$-dimensional case, we need only to find a tame estimate when $t \neq 0$ for the solution $(v^1,\ldots,v^N)$ of
\[ q = \OpeDiv(B \nabla v) \qquad \text{with} \qquad B := \frac{g(\idmap - A^{-1}\nabla u)}{\det A} \OpeTranspose{\left[\OpeComatrix{(A - D^2 u)}\right]}.\] 

First, notice that by integrating with respect to $x_N$, we are reduced to the $N-1$ dimensional case. Therefore, we can proceed by induction on $N$.

So let us assume we already have a tame estimate for $v^1, \ldots, v^{N-1}$. To get an estimate for $v^N = \NotGathered{v}^N$, we will find one for each $\NotGathered{v}^k$, this time by induction on $k$. Since $\NotGathered{v}^1 = v$ satisfies a nice strictly elliptic equation, and thus comes with a tame estimate, we need only to show how to get one for $\NotGathered{v}^{k}$ if we have one for $\NotGathered{v}^1,\ldots,\NotGathered{v}^{k-1}$.

The key lies in the following decomposition of the matrix $B$: for any $k$,
\[ B = B^1 + \frac{1}{\lambda_1} B^2 + \frac{1}{\lambda_1 \lambda_2} B^3 + \ldots + \frac{1}{\lambda_1 \cdots \lambda_{k-2}} B^{k-1} +  \frac{1}{\lambda_1 \cdots \lambda_{k-1}} \NotGathered{B}^{k},\]
where the coefficients $(b^i_{\alpha,\beta})$ of $B^i$ are zero except when $\min(\alpha,\beta)=i$, and where the coefficients $(\NotGathered{b}^{k}_{\alpha,\beta})$ of $\NotGathered{B}^{k}$ are zero except for $\min(\alpha,\beta) \geq k$ :
\[ B^i = \left(
		\begin{array}{ccccc}
			 &   &  &  \\
			  & & b^i_{i,i} & \cdots & b^i_{i,N} \\
			 & &\vdots &  &  \\
			  & & b^i_{N,i}& 
		\end{array}
	\right),
	\qquad \NotGathered{B}^{k} = \left(
		\begin{array}{ccccc}
			 &   &  &  \\
			  & &  \NotGathered{b}^{k}_{k,k} & \cdots &  \NotGathered{b}^{k}_{k,N} \\
			 & &\vdots & \cdots  &  \vdots \\
			  & &  \NotGathered{b}^{k}_{N,k} &\cdots&  \NotGathered{b}^{k}_{N,N}
		\end{array}
	\right),\]
the point being that all the coefficients $b^i_{\alpha,\beta}, \NotGathered{b}^{k}_{\alpha,\beta}$ can be bounded in $\SetCl^{n}$ by the norms of the $u^i$ in $\SetCl^{n+2}$ uniformly in $t$, at least for small $t$.  Indeed, if such a decomposition exists, since 
\[v = v^1 + \lambda_1 v^2 + \ldots + \lambda_1 \cdots \lambda_{i-2} v^{i-1} +  \lambda_1 \cdots \lambda_{i-1} \NotGathered{v}^{i},\]
with $\partial_i v^j = 0$ if $i > j$, which implies $\partial_i v = \lambda_1 \cdots \lambda_{i-1} \partial_i \NotGathered{v}^i$, we have
\[
\OpeDiv(B\nabla v)   = \left[\sum_{i<k} \frac{1}{\lambda_1 \cdots \lambda_{i-1}}\OpeDiv(B^i \nabla v)\right] + \frac{1}{\lambda_1 \cdots \lambda_{k-1}}\OpeDiv(\NotGathered{B}^k \nabla v),
\]
and thus
\begin{equation}\label{EqDivBMatrix}		
\OpeDiv(B\nabla v)   = \left[\sum_{i<k}\OpeDiv(B^i \nabla \NotGathered{v}^i)\right] +\OpeDiv(\NotGathered{B}^k \nabla \NotGathered{v}^k). 
\end{equation}
On the one hand, the matrix $\NotGathered{B}^k$ is symmetric and non-negative, and we define $\Omega_0$ such that
\[ \forall \xi \in \SetR^N, \quad \epsilon\left( \sum_{i\geq k} |\xi_i|^2 \right) \leq \langle \NotGathered{B}^k \xi, \xi \rangle.\]
On the other hand, since 
\[\forall x_1, \ldots, x_{k-1}, \quad \int \cdots \int \NotGathered{v}^k(x_1,\ldots,x_N) \OpeIntd x_k \ldots \OpeIntd x_N = 0,\]
we have a Poincaré inequality:
\[ \left\| \NotGathered{v}^k \right\|^2_{L^2} \leq C \sum_{i \geq k} \left\| \partial_i \NotGathered{v}^k\right\|^2_{L^2}.\]
Therefore, 
\[ \left\| \NotGathered{v}^k \right\|^2_{L^2} \leq \frac{C}{\epsilon} \int \langle \NotGathered{B}^k \nabla  \NotGathered{v}^k, \nabla  \NotGathered{v}^k \rangle \leq \frac{C}{\epsilon} \left\| \OpeDiv(\NotGathered{B}^k \nabla  \NotGathered{v}^k) \right\|_{L^2} \left\| \NotGathered{v}^k \right\|_{L^2}, \]
and this shows how we can deduce a $L^2$ estimate for $\NotGathered{v}^k$ from \eqref{EqDivBMatrix} and a series of estimates for $\NotGathered{v}^i$, $i < k$. Estimates for the norms $H^n$, $n > 0$, easily follow, by the same reasoning as in Section \ref{SecEstimate}.

Thus, all we need is the decomposition
\[ B = B^1 + \frac{1}{\lambda_1} B^2 + \frac{1}{\lambda_1 \lambda_2} B^3 + \ldots  +  \frac{1}{\lambda_1 \cdots \lambda_{N-1}} {B}^{N},\]
with
\[ B^i = \left(
		\begin{array}{ccccc}
			 &   &  &  \\
			  & & b^i_{i,i} & \cdots & b^i_{i,N} \\
			 & &\vdots &  &  \\
			  & & b^i_{N,i}& 
		\end{array}
	\right).\]
Remember that
\[B := \frac{g(\idmap - A^{-1}\nabla u)}{\det A} \OpeTranspose{\left[\OpeComatrix{(A - D^2 u)}\right]},\]
and $\det A = \lambda_1 \left( \lambda_1 \lambda _2 \right) \cdots \left( \lambda_1 \cdots \lambda_{N-1} \right)$, therefore all we have to do is to show how in $\OpeComatrix{(A - D^2 u)}$ we can gather the $\lambda_k$ so as to get the decomposition we seek.
Since $\partial_{i,j} u = \lambda_1 \cdots \lambda_{\max{(i,j)}-1} \partial_{i,j} \NotGathered{u}^{\max{(i,j)}}$,
\begin{multline*}
 \left[\OpeComatrix{(A - D^2 u)}\right]_{i,j} =  \sum_{\substack{\sigma \in \SetPerm_n\\ \sigma(i)=j}} \prod_{\substack{1 \leq k \leq N \\ k \neq i}} (A - D^2 u)_{k,\sigma(k)}  \\
 =  \sum_{\substack{\sigma \in \SetPerm_n\\ \sigma(i)=j}} \prod_{\substack{1 \leq k \leq N \\ k \neq i}} \lambda_1 \cdots \lambda_{\max(k,\sigma(k))-1} \left( \delta_{k,\sigma(k)} - \partial_{k,\sigma(k)} \NotGathered{u}^{\max(k,\sigma(k))}\right).
 \end{multline*}
Thus, if $i \leq j$, we set $\omega_{\alpha,\beta} =  \lambda_{\alpha} \cdots \lambda_{\max(\alpha,\beta)-1} \left( \delta_{\alpha,\beta} - \partial_{\alpha,\beta} \NotGathered{u}^{\max(\alpha,\beta)}\right)$, and then we get
\begin{multline*}
 \left[\OpeComatrix{(A - D^2 u)}\right]_{i,j}
 	=  \sum_{\substack{\sigma \in \SetPerm_n\\ \sigma(i)=j}} \epsilon(\sigma) \prod_{\substack{1 \leq k \leq N \\ k \neq i}} \lambda_1 \cdots \lambda_{k-1} \omega_{k,\sigma(k)} \\
	 =  \sum_{\substack{\sigma \in \SetPerm_n\\ \sigma(i)=j}} \frac{\epsilon(\sigma)}{ \lambda_1 \cdots \lambda_{i-1}} \left[\prod_{1 \leq k \leq N} \lambda_1 \cdots \lambda_{k-1}\right]\left[  \prod_{\substack{1 \leq k \leq N \\ k \neq i}} \omega_{k,\sigma(k)}\right] ,
\end{multline*}
that is to say,
\[  \left[\OpeComatrix{(A - D^2 u)}\right]_{i,j} = \frac{\det A}{\lambda_1 \cdots \lambda_{i-1}} \sum_{\substack{\sigma \in \SetPerm_n\\ \sigma(i)=j}} \epsilon(\sigma) \prod_{\substack{1 \leq k \leq N \\ k \neq i}} \omega_{k,\sigma(k)}. \]
Since we have assume $i \leq j$, this is exactly what we wanted.

\bibliography{references}
\bibliographystyle{bibstyle}

\end{document}